\documentclass{amsart}

\usepackage{amsmath}
\usepackage{amsfonts}
\usepackage{amssymb}
\usepackage{amsthm}
\usepackage{comment}
\usepackage{epsfig}
\usepackage{psfrag}
\usepackage{mathrsfs}
\usepackage{amscd}
\usepackage[all]{xy}
\usepackage{rotating}
\usepackage{lscape}
\usepackage{amsbsy}
\usepackage{verbatim}
\usepackage{moreverb}
\usepackage{mathabx}

\newtheorem{theorem}[subsubsection]{Theorem}
\newtheorem{prop}[subsubsection]{Proposition}
\newtheorem{conjecture}[subsubsection]{Conjecture}
\newtheorem{lemma}[subsubsection]{Lemma}
\newtheorem{cor}[subsubsection]{Corollary}

\newtheorem{definition}[subsubsection]{Definition}

\theoremstyle{remark}
\newtheorem{remark}[subsubsection]{Remark}

\DeclareMathOperator{\colim}{colim}

\DeclareMathOperator{\Hom}{Hom}

\DeclareMathOperator{\Spec}{Spec}
\DeclareMathOperator{\ind}{\on{Ind}}

\DeclareMathOperator{\Fact}{Fact}

\DeclareMathOperator{\Lie}{Lie}

\DeclareMathOperator{\oo}{\infty}

\DeclareMathOperator{\m}{\rm Mod}

\DeclareMathOperator{\alg}{\mathrm{-alg}}
\DeclareMathOperator{\coalg}{\mathrm{-coalg}}

\DeclareMathOperator{\prim}{Prim}

\DeclareMathOperator{\ot}{\otimes}

\DeclareMathOperator{\ch}{\rm ch}
\DeclareMathOperator{\op}{\rm op}
\DeclareMathOperator{\id}{\rm id}

\DeclareMathOperator{\C}{C}

\DeclareMathOperator{\ran}{{R}an}

\newcommand{\ra}{\rightarrow}

\newcommand{\squig}{\rightsquigarrow}

\newcommand{\fib}{\twoheadrightarrow}

\def\cA{\mathcal A}\def\cC{\mathcal C}\def\cD{\mathcal D}
\def\cE{\mathcal E}\def\cF{\mathcal F}

\def\cM{\mathcal M}\def\cO{\mathcal O}\def\cP{\mathcal P}
\def\cR{\mathcal R}
\def\cX{\mathcal X}
\def\cY{\mathcal Y}

\def\RR{\mathbb R}

\def\fD{\mathfrak D}

\numberwithin{equation}{section}

\newcommand{\nc}{\newcommand}
\nc{\renc}{\renewcommand}
\nc{\on}{\operatorname}
\nc{\ssec}{\subsection}
\nc{\sssec}{\subsubsection}

\DeclareMathOperator{\fset}{\on{fSet}}
\DeclareMathOperator{\Sch}{{\rm Sch}}
\nc{\ooCat}{\oo\on{-Cat}}
\nc{\Vect}{\on{Vect}}
\nc{\Funct}{\on{Funct}}
\nc{\Op}{\on{Op}}
\nc{\coOp}{\on{coOp}}
\renc{\Bar}{\on{Bar}}
\nc{\coBar}{\on{Cobar}}
\nc{\triv}{\on{triv}}
\nc{\oblv}{\on{oblv}}
\nc{\Free}{\on{Free}}
\nc{\coFree}{\on{coFree}}
\nc{\nilpcoalg}{\coalg^{\on{nil}}}

\nc{\BG}{\mathbb{G}}
\nc{\BN}{\mathbb{N}}
\nc{\BR}{\mathbb{R}}

\nc{\fg}{\mathfrak g}

\nc{\wh}{\widehat}
\renc{\mod}{\on{-mod}}
\nc{\comod}{\on{-comod}}
\nc{\Comod}{\on{Comod}}
\nc{\nilpcomod}{\on{-comod}^{\on{nil}}}
\nc{\Factalg}{\on{FactAlg}}
\nc{\Factmod}{\on{-FactMod}}

\begin{document}

\title{Chiral Koszul duality}
\author{John Francis and Dennis Gaitsgory}
\address{Department of Mathematics\\Northwestern University\\Evanston, IL 60208-2370}
\email{jnkf@math.northwestern.edu}
\address{Department of Mathematics\\Harvard University\\Cambridge, MA 02138-2901}
\email{gaitsgde@math.harvard.edu}

\begin{abstract} We extend the theory of chiral and factorization algebras, developed for curves by Beilinson and Drinfeld in \cite{bd}, to higher-dimensional varieties. This extension entails the development of the homotopy theory of chiral and factorization structures, in a sense analogous to Quillen's homotopy theory of differential graded Lie algebras. We prove the equivalence 
of higher-dimensional chiral and factorization algebras by embedding factorization algebras into a larger category of chiral commutative coalgebras, then realizing this interrelation as a chiral form of Koszul duality. We apply these techniques to rederive some fundamental results of \cite{bd} on chiral enveloping algebras of $\star$-Lie algebras.
\end{abstract}



\maketitle

\tableofcontents

\section{Introduction}

Beilinson and Drinfeld developed the theory of chiral and factorization (co)algebras on curves in their seminal work, \cite{bd}, as a geometric counterpart of the algebraic theory of vertex algebras. Their theory translated the formulae of operator product expansions in conformal field theory into beautiful algebraic geometry. These two algebraic avatars of conformal field theory at first blush appear quite dissimilar:  A chiral Lie algebra is a D-module on a curve with a type of Lie algebra structure in which one has the extra ability to take the Lie bracket of certain divergent sections; a factorization coalgebra consists of a quasi-coherent sheaf on each configuration space of a curve, with certain compatibilities. One of the conceptually central results of \cite{bd} (Theorem 3.4.9) establishes the equivalence of these two theories of chiral Lie algebras and of factorization coalgebras on algebraic curves.

\medskip

Beilinson and Drinfeld posed several challenges left open by their work: first, to extend their theory above complex dimension 1. Second, in order to sensibly extend the theory to varieties, they observed the necessity of developing the homotopy theory of 
chiral Lie algebras (in a sense analogous to Quillen's homotopy theory of differential graded algebras), a problem of independent interest.\footnote{For instance, the category of chiral Lie algebras on a curve $X$ lacks coproducts, hence it cannot admit a model category structure.} 

\medskip

In this work, we develop just such a homotopy theory of chiral and factorization structures and apply it to prove a generalization of the above theorem of \cite{bd}, to establish an equivalence between chiral Lie algebras and factorization coalgebras on higher-dimensional varieties. The most appealing aspect of this proof is a reconceptualization of the relation between the two: The equivalence between chiral Lie algebras and factorization coalgebras is a form of Koszul duality, in which factorization coalgebras are realized as a full subcategory of a larger category of chiral commutative coalgebras.\footnote{It is for this reason that we take the liberty of adjusting the terminology ``factorization algebra" of \cite{bd} to ``factorization coalgebra," since they are, literally, coalgebras 
rather than algebras with respect to the chiral tensor structure. See Remark \ref{terminology}.} This is a chiral analogue of the duality between Lie algebras and commutative coalgebras that Quillen first developed in his work on rational homotopy theory, \cite{quillenrational}, in which the category of chain complexes, with tensor product, is replaced by that of D-modules on the Ran space, equipped with the chiral tensor product of D-modules. We shall see that despite this apparent increased complexity, chiral Koszul duality is more of a duality than usual Koszul duality, in the sense that the double dual is always a homotopy equivalence, without preconditions.

\medskip

Beilinson and Drinfeld's perspective on chiral versus factorization gave rise to an important new construction, the chiral homology of chiral Lie algebras, a homotopy-theoretic generalization of the space of conformal blocks in conformal field theory. The other primary focus of \cite{bd} was the calculation of chiral homology in several salient examples, including lattice chiral Lie algebras and chiral enveloping algebras of $\star$-Lie algebras. Chiral enveloping algebras are chiral analogues of the usual enveloping algebra of a Lie algebra; they appear in conformal field theory in the construction of affine Kac-Moody chiral Lie algebras, and as such serve as chiral versions of the Lie algebras of loop groups. To illustrate the efficacy of the Koszul duality viewpoint, as an application  we give a conceptual proof of Theorem 4.8.1.1 of \cite{bd}, which expresses the chiral homology of the chiral envelope of a $\star$-Lie algebra $L$ in terms of de Rham cohomology of $L$ itself.

\ssec{Why study chiral algebras?}

Before giving an overview of the contents of this paper, let us offer some general motivation for the study of chiral Lie algebras and factorization coalgebras. Broadly speaking, one can divide the reasons to study them into two classes: local and global.

\sssec{}
Locally, chiral Lie algebras and their representations on curves appear as a general formalism to study 
representation theory of Lie algebras that have a loop component, as well as categories obtained from the
category of Lie algebra modules by various functorial procedures. 

\medskip

For example, consider the Lie algebra of formal Laurent series $\fg(\!(t)\!)$, where $\fg$ is a finite-dimensional Lie algebra. There is a direct route to studying representations of $\fg(\!(t)\!)$, but in which one is required to take into account the topology on $\fg(\!(t)\!)$: This is certainly doable, though it makes homological algebra more cumbersome. However, to then further study those representations that are {\it integrable} (i.e., those that arise from differentiating positive energy representation of the loop group $G(\!(t)\!)$), becomes impracticable from the vantage of topological associative algebras.

\medskip

Another local aspect of the story is the connection between chiral Lie algebras and $\cE_2$-algebras. Via the Riemann-Hilbert correspondence, $\cE_2$-algebras form a full subcategory of chiral Lie algebras on the affine line, consisting of those chiral algebras whose underlying D-module is holonomic with regular singularities. 

\medskip

This perspective allowed one to rediscover chiral Lie algebras in their factorization incarnation
in the work of Schechtman-Varchenko and its elaboration by Bezrukavnikov-Finkelberg-Schechtman 
(see \cite{BFS} and references therein) on the construction of quantum groups via configuration spaces, 
and its relation to the Kazhdan-Lusztig equivalence between quantum groups and Kac-Moody 
representations. 

\medskip

Further, the discovery of factorization coalgebras led to the notion of a factorization category, which appears as a very potent tool for many problems of geometric representation theory (see \cite{Ga} for a brief review).  

\sssec{}
Let us now turn to the global aspects of the theory. For this discussion we assume that $X$ is complete. The overarching reason for the usefulness of chiral Lie algebras is that the procedure of taking chiral homology of chiral Lie algebras/factorization coalgebras is a powerful local-to-global principle. 

\medskip

For example, let $Y$ be a scheme affine over $X$, and suppose one is interested to study the scheme of its
global sections $X\to Y$. According to \cite{bd}, Theorem 4.6.1, this scheme can be described as
$\Spec$ of the chiral homology of a certain chiral algebra. 

\medskip

The above example is ``commutative" in the sense of \cite{bd}, Sect. 4.6. A non-commutative, but
relatively elementary, example of an application of the above local-to-global principle is the 
construction of Hecke eigensheaves in the geometric Langlands program carried out in \cite{bd1}.

\medskip

However, this local-to-global principle can be applied in significantly more sophisticated situations. In particular, it plays a prominent role in the recent advances in the geometric Langlands program, where one applies it in the case of chiral Lie algebra that controls twisted Whittaker sheaves. 

\medskip

We should also remark that the functor of chiral homology on the category of chiral Lie algebras 
bears a strong similarity with the assignment in quantum field theory to a collection of local observables of the value of the corresponding correlation function at a particular configuration of points on a compact space-time. 

\ssec{Contents}

We now review the contents of the paper and state our main results:

\medskip

Throughout the paper we will be working over a ground field $k$ of characteristic $0$. We will be working
with the category $\Sch$ of schemes of finite type over $k$, and for any $Y\in \Sch$ we denote by $\fD(Y)$ the 
stable $\oo$-category of D-modules on $Y$ (see Section  \ref{sss:Dmod} where our conventions
regarding $\fD(Y)$ are explained). 

\medskip

For the duration of the paper we fix $X$ to be a separated scheme of finite type over $k$. 

\sssec{The Ran space}

Our main geometric object of study is the Ran space of $X$, which should be thought of as the ``space of all finite configurations of points of $X$," and the category of D-modules on it. In other words, $\ran X$ is intuitively given by the union 
$\underset{j}\bigcup\, {\rm Conf}_j X$ of the configuration spaces of unordered points in $X$, topologized so that that two points may collide and pass to a different stratum, i.e., so that the map $X^I \ra \underset{j}\bigcup\, {\rm Conf}_j X$ is continuous for each $n$. However, this intuition does not immediately translate into a genuine definition: $\ran X$ does not exist as a scheme or even an ind-scheme, and so the category of D-modules on it is not a priori defined.  

\medskip

To remedy this, we can consider the structure that we would see if $\ran X$ did exist as described: For a D-module $M$ on $\ran X$, we could pull it back to $X^I$ to obtain a new D-module, $M^I$, for each finite set $I$; these D-modules would be subject to certain compatibilities under pullbacks, given a factorization $X^J\ra X^I \ra \ran X$. One should imagine that you can completely recover the D-module $M$ from this compatible family of $M^I$.

\medskip

This intuition gives rise to a formal definition. We define $\ran X$ as a functor from the category opposite that of finite sets to $\Sch$,
namely $I\squig X^I$, and define the $\oo$-category $\fD(\ran X)$ as the limit of $\fD(X^I)$ over finite maps.
I.e., an object $M\in \fD(\ran X)$ is by definition a collection of objects $M^I\in \fD(X^I)$ for each
finite set $I$ and a homotopy equivalence
$$\Delta(\pi)^!(M^I)\simeq M^J,$$
for every surjection $\pi:I\twoheadrightarrow J$, where $\Delta^\pi:X^J\to X^I$ denotes the corresponding
map. 

\medskip

Now, following \cite{bd} we introduce two symmetric monoidal structures on $\fD(\ran X)$. The first one, the $\star$-tensor product, should be thought of as the direct image with respect to the map
$$\on{union}:\ran X\times \ran X\longrightarrow \ran X$$
given by the operation of union of finite sets. I.e., it is convolution with respect to the abelian semi-group structure on $\ran X$.

\medskip

The other symmetric monoidal structure, the chiral tensor product, is the composition
$$\on{union}_*\circ \jmath_*\circ \jmath^*,$$
where $\jmath$ is the open embedding of the locus 
$$\left(\ran X\times \ran X\right)_{\on{disj}}\subset \ran X\times \ran X,$$
corresponding to pairs of finite subsets of $X$ that are disjoint. 

\medskip

In other words, one should think about these two tensor products as follows. For
$M_1,M_2\in \fD(\ran X)$, the fiber of $M_1\otimes^\star M_2$ (resp., $M_1\otimes^{\ch} M_2$)
at a point $\{S\}\in \ran X$, where $S\subset X$ is a finite non-empty subset is
$$\underset{S=S_1\cup S_2}\oplus\, (M_1)_{S_1}\otimes (M_2)_{S_2},$$
where for the $\star$-tensor product the direct sum is taken over all decompositions
as a union of non-empty subsets, and for the chiral tensor product we only take those
summands for which $S_1\cap S_2=\emptyset$.

\sssec{Chiral algebras and factorization coalgebras}

Consider the $\oo$-category $\fD(\ran X)$, endowed with the chiral tensor structure. We can
consider the categories of Lie algebras and commutative coalgebras in it, denoted 
$\Lie\alg^{\ch}(\ran X)$ and $\on{Com}\coalg^{\ch}(\ran X)$, respectively. 

\medskip

Inside $\Lie\alg^{\ch}(\ran X)$ we single out the full subcategory spanned by
objects that, as D-modules on $\ran X$, are supported on the main diagonal $X\subset \ran X$.
We denote this $\oo$-category by $\Lie\alg^{\ch}(X)$.  This is the $\oo$-category of chiral Lie algebras
introduced by \cite{bd}. 

\medskip

Inside $\on{Com}\coalg^{\ch}(\ran X)$ we single out a full subcategory of factorization coalgebras
that we denote by $\on{Com}\coalg^{\ch}_{\on{Fact}}(\ran X)$. We shall also use the notation
$\Fact(X)$ and refer to the objects of this $\oo$-category as factorization D-modules on $\ran X$. 
We shall now indicate its definition: 

\medskip

Let $B$ be a coalgebra in $\fD(\ran X)$. Let $S\subset X$ be a finite subset, and
$S=S_1\sqcup S_2$ be its decomposition as a disjoint union. Then the coalgebra
structure on $B$ defines a map at the level of fibers
\begin{equation} \label{e:naive factorization}
B_S\longrightarrow B_{S_1}\otimes B_{S_2}.
\end{equation}
The factorizability condition is that the above map should be a homotopy equivalence. 

\medskip

Note that the notion of factorization coalgebra can be encoded as an assignment
$$(S\subset X)\squig (B_S\in \Vect_k),$$ 
(and such that this system forms a D-module as $S$ ranges over $\ran X$), and 
a system of homotopy equivalences \eqref{e:naive factorization} that satisfy the natural
compatibility conditions under further partitions of finite sets into disjoint unions.
When written in this form, the notion of factorization D-module looks symmetric
from the algebra/coalgebra perspective.

\sssec{Koszul duality}

Let us now recall the following general construction. Let $\cC$ be a (not necessarily unital)
stable symmetric monoidal $\oo$-category over $k$.  We can consider the $\oo$-category
$\Lie\alg(\cC)$ of Lie algebras in $\cC$, and the $\oo$-category $\on{Com}\coalg(\cC)$ of commutative
coalgebras in $\cC$. These two $\oo$-categories are related by a pair of mutually adjoint functors

\begin{equation} \label{e:Koszul for Lie intro}
\xymatrix{
\Lie\alg(\cC) \ar[rr]<2pt>^{\C} && \on{Com}\coalg(\cC), \ar[ll]<2pt>^{\prim[-1]} \\}
\end{equation}
where the functor $\C$ is the functor computing Lie algebra homology, and the functor $\prim$
is the derived functor of taking primitive elements.

\medskip

The above functors are not in general equivalences of $\oo$-categories. However, they are
for a particular class of tensor $\oo$-categories $\cC$ that we call pro-nilpotent, and the $\oo$-category
$\fD(\ran X)$ is such. This will imply our main result:

\begin{theorem} \label{t:main}
The functors $\C^{\ch}$ and $\prim^{\ch}[-1]$ in $\fD(\ran X)$ define an 
equivalence $$\Lie\alg^{\ch}(\ran X) \simeq \on{Com}\coalg^{\ch}(\ran X)$$
of $\oo$-categories. Moreover, this equivalence induces an equivalence 
between the $\oo$-subcategories of chiral Lie algebras and factorization coalgebras on $X$:
\[\xymatrix{
\Lie\alg^{\ch} ( \ran X) \ar[rr]<2pt>^{\C^{\ch}} && \on{Com}\coalg^{\ch} (\ran X) \ar[ll]<2pt>^{\prim^{\ch}[-1]} \\
\Lie\alg^{\ch}(X) \ar@{_{(}->}[u]\ar[rr]<2pt>&& \mathrm{Fact}(X)\ar@{_{(}->}[u]\ar[ll]<2pt>\\}\]

\end{theorem}

In Section \ref{s:enveloping} we will apply Theorem \ref{t:main} to study chiral Lie algebras
obtained by the taking the chiral envelope of a $\Lie^\star$ algebra. In particular, we will rederive
the \cite{bd} computation of chiral homology of such chiral Lie algebras.

\medskip

In Section \ref{s:modules} we will extend this theorem to include a statement about chiral
modules for chiral Lie algebras.

\sssec{Nilpotence}

Let us comment on the pro-nilpotence condition for a tensor $\oo$-category $\cC$, and why it implies
that the functors \eqref{e:Koszul for Lie intro} equivalences in this case.

\medskip

At least conjecturally, one can modify both sides in \eqref{e:Koszul for Lie intro} to
turn it into an equivalence. Namely, one has to replace the $\oo$-category $\Lie\alg(\cC)$
by its full subcategory $\Lie\alg^{\on{nil}}(\cC)$ consisting of pro-nilpotent
Lie algebras. And one has to replace the $\oo$-category $\on{Com}\coalg(\cC)$ by its full
subcategory $\on{Com}\coalg^{\on{nil}}(\cC)$ consisting of ind-nilpotent commutative
coalgebras. We refer the reader to Section  \ref{ss:turning into equiv} for the
precise formulation of this conjecture. 

\medskip

The main feature of the pro-nilpotence condition on $\cC$ is that in this case the
inclusions
\begin{equation} \label{e:inclusions}
\Lie\alg^{\on{nil}}(\cC)\hookrightarrow \Lie\alg(\cC) \text{ and }
\on{Com}\coalg^{\on{nil}}(\cC)\hookrightarrow \on{Com}\coalg(\cC)
\end{equation}
are equivalences.

\medskip

However, unfortunately, in order to actually prove that \eqref{e:Koszul for Lie intro}
is an equivalence for $\cC=\fD(\ran X)$ we use more than just the above
mentioned fact about the inclusions \eqref{e:inclusions}: Our definition of pro-nilpotence
is quite stringent and explicitly specifies $\cC$ as an inverse limit of $\oo$-categories with vanishing
$n$-fold tensor products.

\ssec{$\oo$-categories}  \label{ss:oo conventions}

\sssec{}  \label{sss:oo-cat}

In this work, we study aspects of the homotopy theory of certain algebro-geometric structures. Classically, such as in the study of chain complexes, a notion of a homotopy theory is provided by the homotopy category, a category modulo some equivalence relation. This notion is very useful for a number of purposes, but it is insufficient for many others -- for instance, differential graded algebras should have a homotopy theory, but it cannot be extracted with any facility from the homotopy category of complexes. Another, richer, notion of a homotopy theory is provided by Quillen's theory model categories, a category equipped with specified types of morphisms: cofibrations, fibrations, and weak equivalences. Quillen's notion is powerful and sufficient for many purposes, but it, in some sense, has more structure than just the homotopy theory. If we were to allow an analogy with linear algebra, the homotopy/triangulated category is like the rank of a module, and a model category is like a module together with a choice of basis: The homotopy theory itself, like the module, is something in between. Further, working with bases can be very useful in algebra, but they only exist if the module is free, and some constructions are easier coordinate-free; similar is true in homotopy theory. 

\medskip

In the present work, this intermediate notion of homotopy theory will be that of an {\it $\oo$-category}. Intuitively, an $\oo$-category consists of the structure of objects, maps, homotopies between maps, homotopies between homotopies, and so forth. Such a structure is provided, for instance, by a category enriched in chain complexes or topological spaces. Topological and DG categories are simple to define, but suffer from technical drawbacks, and we instead use Joyal's {\it quasi-category} model for $\oo$-category theory, where this data is just a particular type of simplicial set, satisfying the weak Kan condition of Boardman-Vogt, \cite{bv}. This theory has been developed in great detail by Joyal in \cite{joyal} and Lurie in \cite{topos} and \cite{dag}, which will be our primary references. The key feature to make note of is that limits, colimits, and functors in the $\oo$-category setting correspond to homotopy limits, homotopy colimits, and derived functors in the setting of DG or model categories. It will be safe to replace the words ``$\oo$-category" by ``topological category" to obtain the intuitive sense of the results in this work, keeping this one proviso in mind.

\medskip

For further motivation for $\oo$-category theory, we refer to Section 2.1 of \cite{qcloops} and, more fundamentally, 
to the first chapter of \cite{topos}.

\sssec{Conventions regarding $\oo$-categories}

We shall use the following notation:

\medskip

\noindent By $\ooCat$ we shall denote the $(\oo,1)$-category of $\oo$-categories. By 
$\ooCat^{\on{st}}$ we shall denote the non-full subcategory of $\ooCat$ consisting
of stable categories and exact functors.

\medskip

\noindent By $\ooCat_{\on{pres}}\subset \ooCat$ we shall denote the full subcategory
consisting of presentable $\oo$-categories. By 
$$\ooCat_{\on{pres},L}\subset \ooCat_{\on{pres}}$$
we shall denote the non-full subcategory where we restrict functors to those commuting with
colimits. The adjoint functor theorem (Corollary 5.5.2.9 of \cite{topos}) says that a functor
between two objects of $\ooCat_{\on{pres}}$ preserves colimits, i.e., is a $1$-morphism in
$\ooCat_{\on{pres},L}$, if and only if it admits a right adjoint. 

\medskip

We let $\ooCat_{\on{pres}}^{\on{st}}$ be the full subcategory of $\ooCat^{\on{st}}$
equal to the preimage of $\ooCat_{\on{pres}}\subset \ooCat$
under the forgetful functor $\ooCat^{\on{st}}\to \ooCat$. We let
$\ooCat_{\on{pres},L}^{\on{st}}$ be the non-full subcategory of $\ooCat_{\on{pres}}^{\on{st}}$
equal to the preimage of $\ooCat_{\on{pres},L}\subset \ooCat_{\on{pres}}$ under the
above forgetful functor.

\medskip

We will also use the notation
$$\ooCat_{\on{pres},\on{cont}}^{\on{st}}:=\ooCat_{\on{pres},L}^{\on{st}},$$
and call its 1-morphisms continuous functors. An exact functor between two
stable presentable categories is continuous if and only if it commutes with
filtered colimits, or, equivalently, with arbitrary direct sums.

\medskip

Using \cite{dag}, Sect. 6.3.1 (and, specifically, Example 6.3.1.22), the category 
$\ooCat_{\on{pres},\on{cont}}^{\on{st}}$ is endowed with a symmetric monoidal
structure under tensor product. 

\medskip

When discussing a monoidal/symmetric monoidal structure on a stable presentable symmetric 
monoidal category, unless specified otherwise, we shall mean a structure of
associative/commutative algebra object in $\ooCat_{\on{pres},\on{cont}}^{\on{st}}$
with respect to the above symmetric monoidal structure on it.

\medskip

For a ground field $k$, we shall denote by $\Vect_k$ the commutative algebra object
of $\ooCat_{\on{pres},\on{cont}}^{\on{st}}$ given by the $\oo$-category associated
to the simplicial category of chain complexes of $k$-vector spaces.

\medskip

\noindent{\it Terminology:} We use the word ``equivalence" in reference to
a functor between $\oo$-categories. We will use the term ``homotopy equivalence"
in reference to a $1$-morphism inside a given $\oo$-category (the notion that for an ordinary
category would be translated as ``isomorphism").

\ssec{D-modules}

\sssec{The naive approach} \label{sss:Dmod}

Let $Y$ be a scheme of finite type. Assume for simplicity that $Y$ is separated. 
We can attach to it a stable $\oo$-category $\fD(Y)$. Namely,
we start with the abelian category $\fD(Y)^\heartsuit$. When $Y$ is smooth, this
is the abelian category of right D-modules over the ring of differential operators
on $Y$; for $Y$ singular one defines this category by locally embedding $Y$
into a smooth scheme and using Kasiwara's theorem (see also \cite{bd}, Sect. 2.1.3).

\medskip

To construct $\fD(Y)$, we consider the DG category of complexes
over $\fD(Y)^\heartsuit$, and following \cite{Dr}, form the DG quotient by the subcategory of acyclic complexes.
It is well-known that to any DG category one can canonically attach a simplicial category,
and $\fD(Y)$ is the $\oo$-category associated to this simplicial category. By construction, the
category $\fD(Y)$ is cocomplete and compactly generated; in particular, it is presentable.

\medskip

The question of functoriality $Y\mapsto \fD(Y)$ is less well understood. With some work
we can extend the above assignment to a functor
\begin{equation} \label{e:star version}
\Sch\longrightarrow \ooCat^{\on{st}}_{\on{pres,cont}},
\end{equation}
such that for a map $f:Y_1\to Y_2$ the resulting functor $\fD(Y_1)\to \fD(Y_2)$
at the level of homotopy categories is given by the D-module push-forward, denoted $f_*$.
We will denote the functor of \eqref{e:star version} by $\fD^*$.

\medskip

One can also extend the assignment $Y\squig \fD(Y)$ differently. Namely, one can construct
a functor
\begin{equation} \label{e:shriek version}
\Sch^{\on{op}}\longrightarrow \ooCat^{\on{st}}_{\on{pres,cont}},
\end{equation}
such that for a map $f:Y_1\to Y_2$ the resulting functor $\fD(Y_2)\to \fD(Y_1)$
at the level of homotopy categories is given by the D-module pullback, denoted $f^!$.
We will denote the functor of \eqref{e:shriek version} by $\fD^!$.

\medskip

However, for most applications that involve $\oo$-categories, considering the above
two functors $\fD^*$ and $\fD^!$ separately is not sufficient. Below we formulate a
version of the formalism of a ``theory of D-modules" which is sufficient for the 
applications that the authors are aware of. 

\begin{remark}
To the best of our knowledge, the construction of the theory of D-modules as formulated 
below does not have a reference in the literature, although many papers on the subject implicitly
assume its existence. We hope, however, that this theory will be documented soon.\footnote{The corresponding theory in a related context of ind-coherent sheaves has been
developed in \cite{IndCoh}.}
\end{remark}

\sssec{The theory of D-modules}

Let $\Sch^{\on{corr}}$ denote the $(1,1)$-category whose objects are schemes of finite type, and morphisms
are correspondences, i.e., for $Y_1,Y_2\in \Sch^{\on{corr}}$, then $\Hom_{\Sch^{\on{corr}}}(Y_1,Y_2)$
is the \emph{groupoid} of diagrams, an element $f$ in which is of the form
\begin{equation} \label{e:corr}
Y_1\overset{f^l}\longleftarrow Z\overset{f^r}\longrightarrow Y_2,
\end{equation}
where maps in this groupoid are defined naturally. For a correspondence as in \eqref{e:corr} we shall
symbolically denote by $(f^l,Z,f^r)$ the corresponding morphism in $\Sch^{\on{corr}}$.

\medskip

The composition of morphisms in $\Sch^{\on{corr}}$
is defined naturally by forming Cartesian products. The unit morphism $Y\to Y$ is one where the maps
$f^l$ and $f^r$ are both isomorphisms. The category $\Sch^{\on{corr}}$ has a natural symmetric monoidal
structure given by products.

\medskip

The category $\Sch^{\on{corr}}$ contains a non-full subcategory denoted $\Sch^*$, equivalent to the usual category
$\Sch$, which has the same objects, but where the morphisms are restricted to have $f^l$ an isomorphism.
We have another non-full subcategory $\Sch^!\subset \Sch^{\on{corr}}$, equivalent to $\Sch^{\on{op}}$, 
which also has the same objects, but where the morphisms are restricted to have $f^r$ an isomorphism. 

\medskip

We assume ``the theory of D-modules" in the following format: We assume having a symmetric monoidal
functor
\begin{equation} \label{e:category of Dmodules}
\fD^{\blackdiamond}:\Sch^{\on{corr}}\to \ooCat^{\on{st}}_{\on{pres,cont}},
\end{equation}
whose value on a scheme $Y$ is the $\oo$-category $\fD(Y)$,
and for a morphism as in \eqref{e:corr} the functor $\fD(Y_1)\to \fD(Y_2)$ is given by $$f\blackdiamond:=(f^r)_*\circ (f^l)^!.$$

\begin{remark}
Modulo homotopy theory, the content of the functor \eqref{e:category of Dmodules} is the
base change theorem: For a Cartesian square in $\Sch$
$$
\CD
Y'  @>{g_Y}>>  Y \\
@V{\pi'}VV   @VV{\pi}V  \\
X'   @>{g_X}>>  X
\endCD
$$
we have a canonical homotopy equivalence $g_X^!\circ \pi_*\simeq \pi'_*\circ g_Y^!$.
\end{remark}

\medskip

Restricting the functor $\fD^\blackdiamond$ to the subcategories $\Sch^*$ and $\Sch^!$, we obtain 
symmetric monoidal functors $\fD^*$ and $\fD^!$ of \eqref{e:star version} and \eqref{e:shriek version},
respectively.

\sssec{}

Let us observe that the theory of D-modules given by \eqref{e:category of Dmodules} encodes
also the standard adjunctions:

\medskip

It follows from the definitions that if $g:Y\to X$ is a locally closed embedding, we have 
a natural homotopy equivalence in $\Hom_{\Sch^{\on{corr}}}(Y,Y)$
$$(\on{id},Y,g)\circ (g,Y,\on{id})\simeq \on{id}_Y,$$
inducing the homotopy equivalence of functors
$$g^!\circ g_*\simeq \on{Id}_{\fD(Y)}.$$

\medskip

If $g$ is a closed embedding $g=\imath$, then the resulting map $\on{Id}_{\fD(Y)}\to \imath^!\circ \imath_*$
is the unit of the $(\imath_*,\imath^!)$ adjunction. 

\medskip

If $g$ is an open embedding $g=\jmath$, then the resulting map $\jmath^!\circ \jmath_*\to \on{Id}_{\fD(Y)}$
is the counit of the $(\jmath^!,\jmath_*)$ adjunction. 

\medskip

\noindent{\it Note on notation:} To be consistent with the notation from \cite{bd}, for
an open embedding $\jmath$, we will often write $\jmath^*$ instead of $\jmath^!$.

\medskip

Thus, the restriction of $\fD^*$ to the non-full subcategory of $\Sch^{\on{open}}\subset \Sch$
with the same objects but open embeddings as morphisms, is a functor
$$\Sch^{\on{open}}\to \ooCat^{\on{st}}_{\on{pres,cont}},$$
obtained from $\fD^!|_{(\Sch^{\on{open}})^{\on{op}}}$ by taking right adjoints. 

\sssec{}

Let now $g:Y\to X$ be an arbitrary separated map, and let $\Delta(Y/X)$ be the diagonal 
$$Y\to Y\underset{X}\times Y.$$ 
From the $(\Delta(Y/X)_*,\Delta(Y/X)^!)$ adjunction above, we obtain a canonical map
$$\on{Id}_{\fD(Y)}\to g^!\circ g_*.$$
When $g$ is proper this map is a unit for the $(g_*,g^!)$ adjunction. 

\medskip

Thus, the restriction of $\fD_*$ to the non-full subcategory of $\Sch^{\on{proper}}\subset \Sch$
with the same objects but proper maps as morphisms, is a functor 
$$\Sch^{\on{proper}}\to \ooCat^{\on{st}}_{\on{pres,cont}},$$
obtained from $\fD^!|_{(\Sch^{\on{proper}})^{\on{op}}}$ by taking left adjoints. 

\sssec{}  \label{sss:! product}

For future use let us note that the functor $$\fD^!:\Sch^{\op}\to \ooCat^{\on{st}}_{\on{pres,cont}}$$
naturally factors through the $(\oo,1)$-category of commutative algebras in 
$\ooCat^{\on{st}}_{\on{pres,cont}}$. Indeed, this structure is induced by the canonical
coalgebra structure on every $Y\in \Sch$ given by the diagonal map.

\ssec{Acknowledgements}  

This work originated in the course of Mike Hopkins' seminar on chiral algebras at Harvard during the fall of 2007, and we thank all the seminar 
participants for their collaboration and fellowship. We especially thank Jacob Lurie for many conversations, during which many of the present 
ideas were jointly conceived, and for numerous explanations on various topics addressed in this paper.\footnote{Unfortunately, Jacob has declined to sign this work as a coauthor.} We warmly thank Sasha Beilinson for conversations on chiral algebras and for his generous encouragement on first hearing of this work. DG thanks Nick Rozenblyum for many helpful discussions.  JF thanks Reimundo Heluani 
for first explaining chiral algebras to him as a first-year graduate student, which led him to \cite{bd} and the comments of 3.3.13 therein.
Finally, we wish to thank the anonymous referee, whose remarks have helped to improve the paper.

\smallskip

JF is supported by an NSF postdoctoral fellowship. JF's travel to Cambridge, during which a draft of this paper was written, was supported by the Midwest Topology Network. DG is supported by NSF grant DMS - 0600903.

\section{Chiral algebras and factorization coalgebras}

\ssec{D-modules on the Ran space}  \label{ss:defn of Ran}

\sssec{}

Let $\fset^{\rm surj}$ denote the category of non-empty finite sets and surjective morphisms. Let $X^{\fset^{\rm surj}}$
denote the functor $(\fset^{\rm surj})^{\on{op}}\to \Sch$ given by $I\squig X^I$. By composing with the
functor $$\fD^!:(\Sch)^{\on{op}}\longrightarrow \ooCat^{\on{st}}_{\on{pres},\on{cont}}$$
we obtain a functor
\begin{equation} \label{e:Ran diag}
\fD^!(X^{\fset^{\rm surj}}):\fset^{\rm surj} \longrightarrow \ooCat^{\on{st}}_{\on{pres},\on{cont}}.
\end{equation}

\medskip

\begin{definition} \label{d:ran}
The $\oo$-category $\fD(\ran X)$ is the limit of the functor in \eqref{e:Ran diag} in $\ooCat^{\on{st}}_{\on{pres},\on{cont}}$.
\end{definition}

\medskip

For a finite set $I$, we will denote by $(\Delta^I)^!$ the tautological functor 
$\fD(\ran X)\to \fD(X^I)$
corresponding to evaluation on $I$. For $I=\on{pt}$, we shall denote
$(\Delta^I)^!$ also by $(\Delta^{\on{main}})^!$. 

\sssec{}

Let us recall the following general paradigm. Let $K$ be a small category, and let $$\Phi:K\to \ooCat^{\on{st}}_{\on{pres}}$$
be a functor. Assume that for every arrow $\alpha:k_1\to k_2$ in $K$, the corresponding functor
$$\Phi_{\alpha}:\Phi_{k_1}\to \Phi_{k_2}$$
admits a left adjoint (which is automatically a $1$-morphism in $\ooCat^{\on{st}}_{\on{pres},\on{cont}}$).

\medskip

Then we can extend the assignment $$i\squig \Phi_i,\,\,(\alpha:k_1\to k_2)\squig (\Phi_\alpha)^{\on{L}}$$
to a functor $\Phi^{\on{L}}:K^{\on{op}}\to \ooCat^{\on{st}}_{\on{pres},\on{cont}}$. Moreover, we have a canonical equivalence
(see e.g., \cite{GL:DG}, Lemma 1.3.3):
\begin{equation} \label{e:limit and colimit}
\underset{K}{\lim}\, \Phi \simeq \underset{K^{\on{op}}}{\colim}\, \Phi^{\on{L}},
\end{equation}
where the colimit is taken in the $(\oo,1)$-category $\ooCat^{\on{st}}_{\on{pres},\on{cont}}$.

\begin{remark} 
Note that the forgetful functor
$$\ooCat^{\on{st}}_{\on{pres},\on{cont}}\to \ooCat^{\on{st}}$$ commutes
with limits, but not with colimits. So, whereas the $\oo$-category in the left-hand side in \eqref{e:limit and colimit}
can be calculated in either $\ooCat^{\on{st}}_{\on{pres},\on{cont}}$ or $\ooCat^{\on{st}}$, it is
crucial that the right-hand side is calculated in $\ooCat^{\on{st}}_{\on{pres},\on{cont}}$.
\end{remark}

\sssec{}

Applying \eqref{e:limit and colimit} to $K=\fset^{\rm surj}$ and $\Phi=\fD^!(X^{\fset^{\rm surj}})$, 
we obtain that $\fD(\ran X)$ can be written as a colimit as follows:
\begin{equation} \label{e:Ran as colimit}
\underset{(\fset^{\rm surj})^{\on{op}}}{\colim}\, \fD^*(X^{\fset^{\rm surj}}).
\end{equation}

Here $\fD^*(X^{\fset^{\rm surj}})$ is the functor $(\fset^{\rm surj})^{\on{op}}\to \ooCat^{\on{st}}_{\on{pres},\on{cont}}$
equal to the composition
$$(\fset^{\rm surj})^{\on{op}} \overset{X^{\fset^{\rm surj}}}\longrightarrow \Sch \overset{\fD^*}\longrightarrow 
\ooCat^{\on{st}}_{\on{pres},\on{cont}}.$$

\medskip

For a finite set $I$, we will denote by $(\Delta^I)_*$ the tautological functor $\fD(X^I)\to \fD(\ran X)$.
By construction, this functor is the left adjoint of $(\Delta^I)^!$. 

\medskip

For $I=\on{pt}$, we will denote
$(\Delta^I)_*$ also by $(\Delta^{\on{main}})_*$. The following is straightforward:

\begin{lemma}
The adjunction $\on{Id}\to (\Delta^{\on{main}})^!\circ (\Delta^{\on{main}})_*$ is a homotopy equivalence.
\end{lemma}

\begin{cor}
The functor $(\Delta^{\on{main}})_*:\fD(X)\to \fD(\ran X)$ is fully faithful.
\end{cor}

\ssec{Symmetric monoidal structures on $\fD(\ran X)$}  \label{ss:defn of conv}

We shall now recall the definition of the $\star$ and chiral symmetric monoidal structures
on $\fD(\ran X)$, borrowed from \cite{bd}, Sect. 3.4.10. 

\medskip

We shall first give a definition based on the formalism of the theory of D-modules 
formulated in Section  \ref{sss:Dmod}. We shall subsequently write it down more
concretely as functors 
\begin{equation} \label{e:J-fold tensor}
\fD(\ran X)^{\otimes J}\to \fD(\ran X)
\end{equation}
for every finite set $J$.

\medskip

Both versions of the definition may be difficult to parse. We refer the reader to Section  \ref{sss:top interp}
where this definition is reinterpreted in the context of sheaves on a topological space, which makes
it more transparent.

\sssec{}

Let us recall the following general paradigm. Let $K$ be a small symmetric monoidal category, and let
$\Psi:K\to \cA$ be a \emph{right lax} symmetric monoidal functor, where $\cA$ is another symmetric 
monoidal category closed under colimits. Then 
$$\underset{K}{\colim}\, \Psi\in \cA$$
is a commutative algebra object in $\cA$.

\sssec{}

We shall apply this to $K:=(\fset^{\rm surj})^{\on{op}}$ and $\cA:=\ooCat^{\on{st}}_{\on{pres},\on{cont}}$,
where $(\fset^{\rm surj})^{\on{op}}$ is viewed as a symmetric monoidal category via the operation of
disjoint union.

\medskip

The functor $\Psi$
will be the composition of $\fD\blackdiamond:\Sch^{\on{corr}}\to \ooCat^{\on{st}}_{\on{pres},\on{cont}}$,
preceded by either of two right lax symmetric monoidal functors:
$$(X^{\fset^{\rm surj}})^\star \text{ and }(X^{\fset^{\rm surj}})^{\ch}:(\fset^{\rm surj})^{\on{op}}\to \Sch^{\on{corr}}.$$

\medskip

We let the functor $(X^{\fset^{\rm surj}})^\star$ be the functor
$$X^{\fset^{\rm surj}}:(\fset^{\rm surj})^{\on{op}}\to \Sch\simeq \Sch^*\hookrightarrow \Sch^{\on{corr}},$$
equipped with a natural symmetric monoidal structure. Note that this functor is not only right lax
monoidal, but actually monoidal.

\sssec{}   \label{sss:defn of monoidal}

The functor $(X^{\fset^{\rm surj}})^{\ch}$ is defined as follows. As a functor
$(\fset^{\rm surj})^{\on{op}}\to \Sch^{\on{corr}}$, it equals $X^{\fset^{\rm surj}}$. However,
the lax symmetric monoidal structure is different:

\medskip

Let $I_J$ be a collection of finite sets, parameterized by another finite set $J$: $j\squig I_j$,
which we can also think of as a surjection
$$\underset{j\in J}\sqcup\, I_j=:I\overset{\pi}\twoheadrightarrow J.$$

Let $U(\pi)$ be the open subset of $X^I$ equal to the locus
$$\{i\squig x_i\in X,\,\,i\in I\,\,|\,\, x_{i_1}\neq x_{i_2} \text{ if } \pi(i_1)\neq \pi(i_2)\},$$
and let
$$\jmath(\pi):U(\pi)\hookrightarrow X^I$$
denote the corresponding open embedding.

\medskip

We define the right lax symmetric monoidal structure on $(X^{\fset^{\rm surj}})^{\ch}$ by letting the arrow
$$\underset{j\in J}\Pi\, X^{I_j}\to X^I\in \Sch^{\on{corr}}$$
be given by the correspondence
$$\underset{j\in J}\Pi\, X^{I_j} \overset{\jmath(\pi)}\longleftarrow U(\pi)
\overset{\jmath(\pi)}\longrightarrow X^I.$$

\ssec{Explicit description of tensor product functors} \label{ss:defn of conv expl}

\sssec{}  

Using the presentation of $\fD(\ran X)$ as a colimit given by \eqref{e:Ran as colimit}, in order 
to define a functor as in \eqref{e:J-fold tensor}, it suffices to define a functor 
$$m_J:
\underset{J}{\underbrace{(\fset^{\rm surj})^{\on{op}}\times...\times  (\fset^{\rm surj})^{\on{op}}}}\longrightarrow (\fset^{\rm surj})^{\on{op}}$$
and a natural transformation between the resulting two functors
$$\underset{J}{\underbrace{(\fset^{\rm surj})^{\on{op}}\times...\times  (\fset^{\rm surj})^{\on{op}}}}\rightrightarrows
\ooCat^{\on{st}}_{\on{pres},\on{cont}}:$$
\begin{equation} \label{e:nat transf for monoidal}
\left(I_J\squig \underset{j\in J}\otimes \fD(X^{I_j})\right)\Rightarrow \left(I_J\squig \fD(X^{m_J(I_J)})\right),
\end{equation}
where we denote by $I_J$ an object of 
$\underset{J}{\underbrace{(\fset^{\rm surj})^{\on{op}}\times...\times  (\fset^{\rm surj})^{\on{op}}}}$
as in Section  \ref{sss:defn of monoidal}.

\medskip

For both monoidal structures, we let $m_J$ to be the functor of disjoint union:
$$I_J\mapsto I:=\underset{j\in J}\sqcup\, I_j.$$ 

\sssec{}

For the $\star$ symmetric monoidal structure, denoted symbolically $\otimes^\star$, we let the natural transformation
of \eqref{e:nat transf for monoidal} to be the external tensor product:
$$\Bigl(M^{I_j}\in \fD(X^{I_j})\Bigr)\squig \left(\underset{j}\boxtimes \,M^{I_j}\in \fD(X^I)\right).$$

\medskip

Note that for objects $M_j\in \fD(\ran X)$, $j\in J$ and a finite set $I$ equipped with a surjection
$\pi:I\twoheadrightarrow J$, there exists a canonical map
\begin{equation} \label{e:approx star product}
\underset{j\in J}\boxtimes\, \Bigl((\Delta^{I_j})^!(M_j)\Bigr)\longrightarrow (\Delta^I)^!\left(\underset{j\in J}{\otimes^{\star}}\, M_j\right).
\end{equation}

\sssec{}

For the chiral symmetric monoidal structure, denoted symboliccally $\otimes^{\ch}$,
we define the natural transformation \eqref{e:nat transf for monoidal} as
$$\left(M^{I_j}\in \fD(X^{I_j})\right)\squig \left(\jmath(\pi)_*\circ \jmath(\pi)^*\bigl(\underset{j}\boxtimes \,M^{I_j}\bigr)\in \fD(X^I)\right).$$

\medskip

Note that for objects $M_j\in \fD(\ran X)$, $j\in J$ and a finite set $I$ equipped with a surjection
$\pi:I\twoheadrightarrow J$, there exists a canonical map
\begin{equation} \label{e:approx ch product}
\jmath(\pi)_*\circ \jmath(\pi)^*\left(\underset{j\in J}\boxtimes\, (\Delta^{I_j})^!(M_j)\right)
\longrightarrow (\Delta^I)^!\left(\underset{j\in J}{\otimes^{\ch}}\, M_j\right).
\end{equation}

The following assertion results from the definitions:

\begin{lemma}  \label{l:ch product on open}
For $M_j\in \fD(\ran X)$, $j\in J$ and $I$ as above,
the resulting map
$$\underset{\pi}\oplus\, \jmath(\pi)_*\circ \jmath(\pi)^*\left(\underset{j\in J}\boxtimes\, (\Delta^{I_j})^!(M_j)\right)
\longrightarrow 
(\Delta^I)^!\left(\underset{j\in J}{\otimes^{\ch}}\, M_j\right).$$
is a homotopy equivalence, where the direct sum is taken over all surjections $\pi:I\twoheadrightarrow J$. 
\end{lemma}

\ssec{Chiral Lie algebras and factorization coalgebras}  \label{ss:chiral and factorization algs}

We now define the $\oo$-categories which will be our primary objects of study.

\begin{definition} We let $\Lie\alg^{\ch}(\ran X)$ and $\Lie\alg^{\star}(\ran X)$
be the $\oo$-categories of Lie algebras in the $\oo$-category $\fD(\ran X)$ equipped with the chiral 
and $\star$ symmetric monoidal structure, respectively.
\end{definition}

\begin{definition}
The $\oo$-category of chiral Lie and $\star$-Lie algebras on $X$ are the full $\oo$-subcategories 
$$\Lie\alg^{\ch}(X)\subset \Lie\alg^{\ch}(\ran X) \text{ and } \Lie\alg^{\star}(X)\subset \Lie\alg^{\star}(\ran X),$$
respectively, spanned by objects 
for which the underlying D-module is supported on $X$, i.e., it lies in the essential image of the functor 
$(\Delta^{\on{main}})_*: \fD(X) \ra \fD(\ran X)$.
\end{definition}

\begin{remark}
Our names for the above objects are slightly different from those in \cite{bd}: What they call
a ``chiral algebra" we call a ``chiral Lie algebra on $X$"; what they call a ``$\Lie^\star$-algebra"
we call a ``$\star$-Lie algebra on $X$."
\end{remark}

\begin{remark}
Throughout this text we will be working with \emph{non-unital} chiral Lie algebras. The precise
relation between non-unital chiral Lie algebras and unital ones will be discussed in another
publication. See also Remark \ref{r:unital chiral homology}.
\end{remark}

\medskip

On the coalgebraic side, we have:

\begin{definition} $\on{Com}\coalg^{\ch}(\ran X)$ is the $\oo$-category of (nonunital) chiral commutative coalgebras for 
the chiral monoidal structure on $\fD(\ran X)$. 
\end{definition}

\sssec{Factorization}

Let $\pi:I\twoheadrightarrow J$ be a surjection of finite sets. If $B\in \on{Com}\coalg^{\ch}(\ran X)$, from 
Lemma \ref{l:ch product on open} we obtain a map
$$(\Delta^{J})^!(B)\longrightarrow \jmath(\pi)_*\circ \jmath(\pi)^*\left(\underset{j\in J}\boxtimes\, (\Delta^{I_j})^!(B)\right)$$
and by adjunction a map 
\begin{equation} \label{e:factorization map}
\jmath(\pi)^*\Bigl((\Delta^J)^!(B)\Bigr)\longrightarrow \jmath(\pi)^*\left(\underset{j\in J}\boxtimes\, (\Delta^{I_j})^!(B)\right).
\end{equation}

\begin{definition} We say that $B\in \on{Com}\coalg^{\ch}(\ran X)$ is a factorization coalgebra if the maps \eqref{e:factorization map}
are homotopy equivalences for all $I$ and $\pi$.
\end{definition}

We define $\on{Com}\coalg^{\ch}_{\on{Fact}}(\ran X)$ to be the full subcategory of 
$\on{Com}\coalg^{\ch}(\ran X)$ spanned by factorization coalgebras. We shall also use the notation
$$\Fact(X):=\on{Com}\coalg^{\ch}_{\on{Fact}}(\ran X).$$ We conclude this subsection with several remarks.

\begin{remark}\label{terminology}
In \cite{bd}, Sect. 3.4.4, the above category $\Fact(X)$ is denoted ${\mathcal {FA}}(X)$,
and its objects are referred to as factorization algebras. Our realization of this category as the
full subcategory of a certain category of coalgebras rather than algebras is Verdier-biased: The latter
would have also been possible if the functors $\jmath(\pi)_!$ had been defined on all of $\fD(U(\pi))$, and not only
on the holonomic subcategory. However, putting the definition of \cite{bd} in the $\oo$-categorical
framework, one can give a Verdier self-dual definition of $\Fact(X)$ by requiring a homotopy-coherent system of homotopy equivalences \eqref{e:factorization map}. For that reason it seems most preferable to term objects of $\Fact(X)$ as ``factorization D-modules."
\end{remark}

\begin{remark} The $\oo$-categories $\Lie\alg^{\ch}(\ran X)$ and $\fD(X)$ are both presentable $\oo$-categories, and they 
both can be made equivalent to the simplicial nerve of model categories. Their intersection, the $\oo$-category of chiral 
algebras $\Lie\alg^{\ch}(X)$, is however not presentable: It fails, for instance, to have coproducts. As a consequence, 
the $\oo$-category of chiral Lie algebras does not arise as the simplicial nerve of a model category. The same holds true on the 
coalgebra side and for $\Fact(X)$.

\end{remark}

\begin{remark} A chiral commutative coalgebra may be thought of as a {\it lax} factorization D-module, i.e., a D-module for 
which there are given the factorizing structure maps (as in \eqref{e:factorization map}),
but which are no longer necessarily homotopy equivalences. The factorization property is closely related to locality in quantum 
field theory, so one might think of general chiral commutative coalgebras as related to field theories in which the condition of locality is weakened. General chiral commutative coalgebras are thus unlikely to be especially physically compelling, but it is still convenient to allow for this mathematical generalization.
\end{remark}

\ssec{Variant: the topological context} 

\sssec{}

In this subsection we let $X$ be a Hausdorff locally compact topological space. We consider the functor
$$I\squig X^I$$
from $(\fset^{\rm surj})^{\on{op}}$ to the category $\on{Top}^{\on{l.c.}}_{\on{cl}}$
of Hausdorff locally compact topological spaces
and maps that are closed embeddings. 

\medskip

Consider the functor 
$$\on{Shv}^!:(\on{Top}^{\on{l.c.}}_{\on{cl}})^{\on{op}}\longrightarrow \ooCat^{\on{st}}_{\on{pres}}$$
that assigns to $\cY\in \on{Top}^{\on{l.c.}}_{\on{cl}}$ the $\oo$-category $\on{Shv}(\cY)$, and
for a closed embedding $f:\cY_1\to \cY_2$ the corresponding functor $f^!$. (Here 
$\on{Shv}(-)$ stands for $\oo$-category of sheaves of $k$-vector spaces, where $k$ is a field
of characteristic $0$.) 

\medskip

Composing, we obtain the functor
$$\on{Shv}^!(X^{\fset^{\rm surj}}):\fset^{\rm surj}\longrightarrow \ooCat^{\on{st}}_{\on{pres}},$$
and we set
$$\on{Shv}(\ran X):=\underset{\fset^{\rm surj}}{\lim}\, \on{Shv}^!(X^{\fset^{\rm surj}}).$$

\medskip

The constructions of Sections \ref{ss:defn of conv} and \ref{ss:chiral and factorization algs}
go through in the present context. In particular, we obtain two symmetric monoidal structures
on $\on{Shv}(\ran X)$, and the notions of chiral Lie algebra, $\star$-Lie algebra and factorization
coalgebra.

\medskip

The analog of Theorem \ref{t:main} goes through for $\on{Shv}(\ran X)$ with no modification.

\sssec{} Let $\cR(X)$ be the topological space defined as in \cite{bd}, Sect. 3.4.1. We have pair
of adjoint functors
\begin{equation} \label{e:two versions of Ran}
\on{Shv}(\ran X)\rightleftarrows \on{Shv}(\cR(X)).
\end{equation}

According to {\it loc.cit.}, Sect. 4.2.4 we have:

\begin{lemma}  \label{l:two versions of Ran}
The functor $\on{Shv}(\ran X)\to \on{Shv}(\cR(X))$ is fully faithful.
\end{lemma}

\sssec{}  \label{sss:top interp}

Let us interpret the $\star$ and chiral symmetric monoidal structures on $\on{Shv}(\ran X)$
in terms of Lemma \ref{l:two versions of Ran}:

\medskip

Following \cite{bd}, Sect. 3.4.1(iii), the topological space $\cR(X)$ is a commutative semigroup
with respect to the operation denoted ``$\on{union}$" (which corresponds to the operation of
taking the union of finite subsets of $X$). 

\medskip

The $\star$-monoidal structure is induced by the above semigroup structure on $\cR(X)$ by
means of the functor of direct image:
$$\on{union}_*:\on{Shv}(\cR(X))^{\otimes I}\longrightarrow \on{Shv}(\cR(X)).$$

\medskip

To describe the chiral symmetric monoidal structure, we note that for a finite set $I$,
the product $\cR(X)^I$ contains an open subset 
$$(\cR(X)^I)_{\on{disj}}\overset{\jmath^I}\hookrightarrow  \cR(X)^I,$$
corresponding to $I$-tuples of finite subsets of $X$ are are pairwise disjoint.

\medskip

The chiral symmetric monoidal structures is given by the functor
$$(i\squig \cF_i\in \on{Shv}(\cR(X)),\,i\in I)\squig
\on{union}_*\left((\jmath^I)_*\circ (\jmath^I)^*(\underset{i}\boxtimes\, \cF_i)\right)\in \on{Shv}(\cR(X)).$$

\medskip

It follows easily from the definitions that the adjoint functors in Lemma \ref{l:two versions of Ran}
intertwine the corresponding symmetric monoidal structures on $\on{Shv}(\ran X)$ and 
$\on{Shv}(\cR(X))$.

\sssec{}

Finally, let us remark how the notion of factorization  coalgebra in $\on{Shv}(\ran X)$ relates
to that of $\cE_n$-algebra:

\medskip

Let us take $X=\BR^n$. As was communicated to us by Lurie, one has the following
assertion:

\begin{theorem}
The $\oo$-category of translation-equivariant factorization  coalgebras in $\on{Shv}(\ran \RR^n)$
is equivalent to that of $\cE_n$-coalgebras over $k$.
\end{theorem}

\begin{remark}
This theorem does not formally follow from Theorem 5.3.4.10 of \cite{dag}: One can show
that for $X$ being a manifold, the category $\on{Shv}^!(\ran X)$ is equivalent to the category
of cosheaves on $\cR(X)$ in the colimit topology rather than the topology in which the 
theorem in \cite{dag} is proved. However, according to Lurie, the above result holds for
the colimit topology as well.

\end{remark}

\begin{remark} Based on the previous remark, we can view the theory of chiral Lie algebras 
studied in this paper as an algebro-geometric analogue of the theory of $\cE_n$-algebras.
Recall now that on the category of $\cE_n$-algebras there is a contravariant $\cE_n$-Koszul
duality functor, introduced in \cite{getzlerjones}. 

\medskip

We should emphasize that chiral Koszul duality studied in this paper is \emph{totally unrelated}
to the $\cE_n$-Koszul duality, either technically or conceptually. 

\medskip

However, we should add that the $\cE_n$-Koszul duality does have an interpretation in the factorization
setting as a form of Verdier duality of (co)sheaves on the Ran space, as is discussed in \cite{fact}. 
In the latter incarnation, $\cE_n$-Koszul duality has an analogue in the algebro-geometric context of 
chiral Lie algebras/factorization D-modules, which we hope to discuss in another publication.

\end{remark}

\section{Algebras and coalgebras over (co)operads: recollections}

This section is included for the reader's convenience. None of the results stated here are original.
The general reference for operads and algebras over them in the $\oo$-category framework is \cite{dag}, 
Chapters 2 and 3.

\ssec{Operads}

\sssec{} 

Let $\cX$ be a stable presentable symmetric monoidal $\oo$-category. Let $\cX^\Sigma$ denote the $\oo$-category of
symmetric sequences in $\cX$. I.e., objects of $\cX^\Sigma$ are collections $\cO=\{\cO(n),\,n\geq 1\}$,
where each $\cO(n)$ is an object of $\cX$ acted on by the symmetric group $\Sigma_n$. 

\medskip

The $\oo$-category $\cX^\Sigma$ has a natural monoidal structure. A convenient way to think
about this monoidal structure is the following:

\medskip

We have a natural functor $\cX^\Sigma\to \Funct(\cX,\cX)$:
\begin{equation} \label{e:action of sym seq}
(\cO=\{\cO(n)\})\squig \left(x\squig \underset{n\geq 1}\sqcup\, (\cO(n)\otimes x^{\otimes n})_{\Sigma_n}\right).
\end{equation}
The monoidal structure on $\cX^\Sigma$ is designed so that the functor in \eqref{e:action of sym seq}
is monoidal.

\begin{definition} 
The $\oo$-category ${\rm Op}(\cX)$ (resp., ${\rm coOp}(\cX)$) of augmented operads (resp., cooperads)
in $\cX$ is that of augmented associative algebras (resp., coalgebras) in $\cX^\Sigma$ with respect to the above
monoidal structure.
\end{definition}

\sssec{} 

We have a pair of adjoint functors  
\begin{equation} \label{e:Bar and coBar}
\Bar:\Op(\cX)\rightleftarrows \coOp(\cX):\coBar
\end{equation} see \cite{getzlerjones}, \cite{koszul}, \cite{ching}.

\medskip

In fact, the above pair of adjoint functors is a particular case of the adjunction
between augmented associative algebras and augmented associative coalgebras,
i.e., of one reviewed in Section \ref{ss:Koszul duality} for $\cO$ being the associative
operad, when we take our ambient monoidal category to be $\cX^{\Sigma}$:

\medskip

In the case of the associative operad, the ambient category needs to be just monoidal,
not symmetric monoidal, and neither does it need to be stable. We only need the monoidal
operation to distribute over sifted colimits in each variable.\footnote{We recall that an index category $I$ is called \emph{sifted} if the diagonal functor $I\to I\times I$ is homotopy cofinal, see \cite{topos}, Definition 5.5.8.1. Filtered categories and $\Delta^{\op}$, the opposite of the simplicial indexing category, are the essential examples.}

\medskip

\begin{definition}
An operad $\cO\in \Op(\cX)$ is derived Koszul if the adjunction map
$$\cO\longrightarrow \coBar\circ \Bar(\cO)$$
is a homotopy equivalence.
\end{definition}

\begin{remark} \label{r:Koszul for operads}
Any Koszul operad in chain complexes in $\cX=\Vect_k$, in the original sense of Ginzburg-Kapranov \cite{koszul}, is derived Koszul 
in the above sense. 

\medskip

In fact, any augmented operad for which $1_{\cX}\to \cO(1)$ is a homotopy equivalence is derived Koszul. 
In particular, the Lie operad is derived Koszul (and this is true even for the Lie operad in spectra, see, e.g., \cite{ching}).  
\end{remark}

\ssec{Algebras over an operad}  \label{ss:alg and mod}

\sssec{}

Let $\cX$ be as above. Let $\cC$ be a (not necessarily unital)
stable presentable symmetric monoidal $\oo$-category compatibly tensored over $\cX$,
i.e., $\cC$ is a commutative algebra object in the $(\oo,1)$-category of $\cX$-modules
in $\ooCat^{\on{st}}_{\on{pres},\on{cont}}$. 

\medskip

Formula \eqref{e:action of sym seq} (applied now to $x$ being an object of $\cC$ rather than $\cX$) defines
an action of the monoidal category $\cX^\Sigma$ acts on $\cC$. Hence, an 
operad $\cO$ (resp., cooperad $\cP$) in $\cX$ defines a monad
$\on{T}_{\cO}:\cC\to \cC$  (resp., a comonad $\on{S}_{\cP}:\cC\to \cC$).

\begin{definition}
For an operad $\cO\in \Op(\cX)$, the $\oo$-category $\cO\alg(\cC)$ of (non-unital) $\cO$-algebras 
in $\cC$ is the $\oo$-category of 
$\on{T}_{\cO}$-modules in $\cC$. 
\end{definition}

\begin{remark} The preceding definition of $\cO\alg(\cC)$ is equivalent to that given by Lurie in the case where $\cX$ is the 
$\oo$-category of topological spaces. We adopt the above definition in order to accommodate the definition of Lie algebras 
in a symmetric monoidal $\oo$-category, since $\Lie$ is an operad in $k$-modules, but not in spaces.
\end{remark}

\begin{definition}
For a cooperad $\cP\in \coOp(\cX)$, the $\oo$-category $\cP\nilpcoalg_{\on{d.p.}}(\cC)$
of ind-nilpotent $\cP$-coalgebras in $\cC$ is the $\oo$-category of 
$\on{S}_{\cP}$-comodules in $\cC$. 
\end{definition}

\begin{remark}
We shall introduce the category of ``all" (i.e., not necessarily ind-nilpotent)
$\cP$-coalgebras in Section \ref{ss:all coalgebras}. In {\it loc.cit.} it will also become
clear why we use the terminology ``ind-nilpotent" for $\cP$-coalgebras in $\cC$.

\medskip

The subscript ``d.p." in $\cP\nilpcoalg_{\on{d.p.}}(\cC)$ stands for ``divided powers." Again, we refer the reader
to Section \ref{ss:all coalgebras} where the reason for this notation will become clear. 

\end{remark}

\sssec{}

Let $\cO$ be an object of $\Op(\cX)$. Let $\oblv_\cO$ denote the tautological forgetful functor
$$\cO\alg(\cC)\to \cC.$$
The functor $\oblv_\cO$ commutes with limits, and with sifted colimits.\footnote{The siftedness condition is used
as follows: For a monoidal category $\cC$ in which tensor products distribute over colimits, the functor of
$n$-th tensor power $c\squig c^{\otimes n}$ distributes over sifted colimits.} 
Let $$\Free_\cO:\cC\to \cO\alg(\cC)$$
denote its left adjoint. 

\medskip

In addition, the augmentation on $\cO$ defines the functor
$$\triv_\cO:\cC\to \cO\alg(\cC).$$
The functor $\triv_\cO$ commutes with both limits and colimits. 

\sssec{}

Let $\cP$ be an object of $\coOp(\cX)$. Let $\oblv_{\cP}$ denote the forgetful functor
$$\cP\nilpcoalg_{\on{d.p.}}(\cC)\to \cC.$$
The functor $\oblv_{\cP}$ commutes with colimits. We let 
$$\coFree_{\cP}:\cC\to \cP\nilpcoalg_{\on{d.p.}}(\cC)$$
denote its right adjoint. 

\medskip

In addition, the augmentation on $\cP$ defines the functor 
$$\triv_{\cP}:\cC\to \cP\nilpcoalg_{\on{d.p.}}(\cC).$$
The functor $\triv_{\cP}$ commutes colimits. 

\medskip

If the cooperad $\cP$ has the property that for every $n$, the functor $c\squig \cP(n)\otimes c$
distributes over limits, then the functor $\triv_{\cP}$ commutes with sifted limits. 

\medskip

The above condition on $\cP$ is satisfied in many cases of interest: e.g., if $\cX=\Vect_k$
and all $\cP(n)$ are (bounded complexes of) finite-dimensional vector spaces.

\ssec{Koszul duality functors}  \label{ss:Koszul duality}

\sssec{}

For $\cO\in \Op(\cX)$, we now consider the left adjoint of the functor $\triv_\cO$, which we denote
$$\Bar_\cO: \cO\alg(\cC)\to \cC.$$

\begin{remark}
At the classical level, since the multiplication on 
$\triv_\cO(M)$ is trivial, any map $A\ra \triv_{\cO}(M)$ must send to zero any element $a\in A$, which is decomposable, i.e., a multiple of two or more elements (e.g., $a= f\cdot a' \cdot a''$, for $a', a''$ in $A$ and $f$ in $\cO(2)$). Consequently, the left adjoint assigns to an $\cO$-algebra $A$ the indecomposables of $A$, the quotient of $A$ by the decomposable elements. In the instance of classical commutative algebra, this quotient is isomorphic to the cotangent space of the associated pointed affine scheme, so one can geometrically imagine the indecomposables as forming an operadic version of the cotangent space. Returning to the homotopy theory, the left adjoint of $\triv_\cO$ can be formed in the model category setting as a derived functor of indecomposables, where one resolves an $\cO$-algebra and takes indecomposables in the resolution.
\end{remark}

\begin{remark}
The reason for notation $\Bar_\cO$ is the following. Let $\cA$ be a monoidal $\oo$-category and $\cM$ a module
category. Then, under some mild hypothesis on $\cA$ and $\cM$, for associative algebras $R,R'\in A$ and a 
homomorphism $R \ra R'$, the functor left adjoint to the forgetful functor $\m_{R'}(\cM) \ra\m_R(\cM)$ 
exists, and is computed as the geometric
realization of simplicial object $\Bar(R',R,-)_\bullet$, called the bar-construction, i.e.,
$$\Bar(R',R,-):=|\Bar(R',R,-)_\bullet|$$
(see \cite{dag} or \cite{thez} for a more extended explanation in the context of $\oo$-categories).
Here we take $\cA=\cX^\Sigma$, $\cM=\cC$, $R=\cO$ and $R'=1$.
\end{remark}

\medskip

The definition of the Koszul dual cooperad as the associative coalgebra in $\cX^\Sigma$ Koszul dual to $\cO$ yields:

\begin{lemma} \label{l:identification of comonad}
There is a natural homotopy equivalence of comonads acting on $\cC$:
$$\Bar_\cO\circ \triv_\cO\simeq \on{S}_{\cO^\vee},$$
where $\cO^\vee:=\Bar(\cO)$. 
\end{lemma}

The general theory of monads,\footnote{Some of this theory is summarized in Section \ref{ss:monad recollections} and
the relevant fact for the next corollary specifically in Section \ref{sss:best approx}.} implies:

\begin{cor} \label{c:identification of comonad}
The functor $\Bar_\cO:\cO\alg(\cC)\to \cC$ factors as
$$\cO\alg(\cC)\overset{\Bar^{\rm{enh}}_\cO}\longrightarrow \cO^\vee\nilpcoalg_{\on{d.p.}}(\cC)
\overset{\oblv_{\cO^\vee}}\longrightarrow \cC$$
for a canonically defined functor 
$\Bar^{\rm{enh}}_\cO:\cO\alg(\cC)\to \cO^\vee\nilpcoalg_{\on{d.p.}}(\cC)$.
\end{cor} 

Since the functor $\Bar_\cO$, being a left adjoint, commutes with colimits, and since
$\oblv_{\cO^\vee}$ commutes with colimits and is conservative, we obtain that the
functor $\Bar^{\rm{enh}}_\cO$ also commutes with colimits.

\sssec{}

We can depict the resulting commutative diagrams of functors as follows:

\begin{gather}  \label{d:O 1st}
\xy
(-30,20)*+{\cO\alg(\cC)}="X";
(30,20)*+{\cO^\vee\nilpcoalg_{\on{d.p.}}(\cC)}="Y";
(0,0)*+{\cC}="Z";
{\ar@{->}^{\Bar^{\rm{enh}}_\cO} "X";"Y"};
{\ar@{->}_{\Bar_\cO} "X";"Z"};
{\ar@{->}^{\oblv_{\cO^\vee}} "Y";"Z"};
\endxy
\end{gather}
and
\begin{gather}  \label{d:O 2nd}
\xy
(-30,20)*+{\cO\alg(\cC)}="X";
(30,20)*+{\cO^\vee\nilpcoalg_{\on{d.p.}}(\cC)}="Y";
(0,0)*+{\cC.}="Z";
{\ar@{->}^{\Bar^{\rm{enh}}_\cO} "X";"Y"};
{\ar@{<-}_{\triv_\cO} "X";"Z"};
{\ar@{<-}^{\coFree_{\cO^\vee}} "Y";"Z"};
\endxy
\end{gather}

\begin{remark} The relative ease of construction for the above diagram is one the great virtues of $\oo$-category theory. In the setting of model categories, one in general loses the strict monad structure on an adjunction when one passes to derived functors: For example, there is a coherence problem to solve in constructing a coalgebra structure on, say, the bar construction $k\ot_A k$ of an augmented algebra at the chain level, \cite{moore}.

\end{remark}

We have also another commutative diagram, namely:
\begin{gather}  \label{d:O 3rd}
\xy
(-30,20)*+{\cO\alg(\cC)}="X";
(30,20)*+{\cO^\vee\nilpcoalg_{\on{d.p.}}(\cC)}="Y";
(0,0)*+{\cC.}="Z";
{\ar@{->}^{\Bar^{\rm{enh}}_\cO} "X";"Y"};
{\ar@{->}^{\rm{Free}_\cO} "Z";"X"};
{\ar@{->}_{\triv_{\cO^\vee}} "Z";"Y"};
\endxy
\end{gather}

\sssec{}

Let $\cP$ be an object of $\coOp(\cX)$, and consider the right adjoint 
of the functor $\triv_{\cP}$:
$$\coBar_{\cP}:\cP\nilpcoalg_{\on{d.p.}}(\cC)\to \cC.$$

As in Lemma \ref{l:identification of comonad} we have:

\begin{lemma} \label{l:identification of monad}
There is a canonical homomorphism of monads
$$\on{T}_{\cP{}^\vee}\to \coBar_{\cP}\circ \triv_{\cP},$$
where $\cP{}^\vee:=\coBar(\cP)$. 
\end{lemma}

\begin{remark}
Unlike Lemma \ref{l:identification of comonad}, the map in the above lemma is no longer
a homotopy equivalence, since the action of $\cX^\Sigma$ on $\cC$ does not commute with
totalizations.
\end{remark} 

\begin{cor} \label{c:identification of monad}
The functor $\coBar_{\cP}:\cP\nilpcoalg_{\on{d.p.}}(\cC)\to \cC$ factors as
$$\cP\nilpcoalg_{\on{d.p.}}(\cC)\overset{\coBar^{\rm{enh}}_{\cP}}\longrightarrow \cP{}^\vee\alg(\cC)
\overset{\oblv_{\cP}}\longrightarrow \cC$$
for a canonically defined functor 
$\coBar^{\rm{enh}}_{\cP}:\cP\nilpcoalg_{\on{d.p.}}(\cC)\to \cP{}^\vee\alg(\cC)$.
\end{cor} 

We can depict the resulting commutative diagram of functors as follows:

\begin{gather}  
\xy
(-30,20)*+{\cP{}^\vee\alg(\cC)}="X";
(30,20)*+{\cP\nilpcoalg_{\on{d.p.}}(\cC)}="Y";
(0,0)*+{\cC.}="Z";
{\ar@{<-}^{\coBar^{\rm{enh}}_{\cP}} "X";"Y"};
{\ar@{->}_{\oblv_\cO} "X";"Z"};
{\ar@{->}^{\coBar_{\cP}} "Y";"Z"};
\endxy
\end{gather}

We have also another commutative diagram, namely:
\begin{gather}  
\xy
(-30,20)*+{\cP{}^\vee\alg(\cC)}="X";
(30,20)*+{\cP\nilpcoalg_{\on{d.p.}}(\cC)}="Y";
(0,0)*+{\cC.}="Z";
{\ar@{<-}^{\coBar^{\rm{enh}}_{\cP}} "X";"Y"};
{\ar@{<-}_{\rm{triv}_\cO} "X";"Z"};
{\ar@{<-}^{\rm{coFree}_{\cO^\vee}} "Y";"Z"};
\endxy
\end{gather}

\sssec{}

Combining Lemmas \ref{l:identification of comonad} and \ref{l:identification of monad}, we obtain:

\medskip 

For $\cO$ and $\cP$ as above, let us be given a map $\cO^\vee\to \cP$, or equivalently, a map
$\cO\to \cP{}^\vee$. These maps define functors
$$\cP{}^\vee\alg(\cC)\to \cO\alg(\cC) \text{ and } \cO^\vee\nilpcoalg_{\on{d.p.}}(\cC)\to \cP\nilpcoalg_{\on{d.p.}}(\cC).$$

\begin{cor} \label{c:enhanced adjunction}
The composed functors
$$\cO\alg(\cC)\overset{\Bar^{\rm{enh}}_\cO}\longrightarrow \cO^\vee\nilpcoalg_{\on{d.p.}}(\cC)\to \cP\nilpcoalg_{\on{d.p.}}(\cC)$$
and 
$$\cO\alg(\cC) \leftarrow \coBar(\cP)\alg(\cC) \overset{\coBar^{\rm{enh}}_{\cP}}\longleftarrow \cP\nilpcoalg_{\on{d.p.}}(\cC)$$
are naturally mutually adjoint.
\end{cor}

\ssec{Turning Koszul duality into an equivalence}  \label{ss:turning into equiv}

\sssec{}

Suppose that the operad $\cO$ is derived Koszul. From the above discussion obtain a pair of adjoint functors:

\begin{equation} \label{e:Koszul adj}
\Bar^{\rm{enh}}_\cO:\cO\alg(\cC)\rightleftarrows \cO^\vee\nilpcoalg_{\on{d.p.}}(\cC):\coBar^{\rm{enh}}_{\cO^\vee}.
\end{equation}

The above adjunction is in general not an equivalence. We shall now describe a procedure how to modify
the left-hand side to (conjecturally) turn it into an equivalence.

\sssec{}  \label{sss:completion}

Let us call an $\cO$-algebra $A$ nilpotent, if there exists an integer $n$, such that the maps
$$\cO(n')\otimes A^{\otimes n'}\to A$$
are null-homotopic for $n'\geq n$.

\begin{definition}
An $\cO$-algebra $A$ is pro-nilpotent if it is equivalent to a limit of nilpotent $A$-algebras.
\end{definition}

Let $\cO\alg^{\on{nil}}(\cC)\subset \cO\alg(\cC)$ denote the full subcategory spanned
by pro-nilpotent algebras. 

\medskip

It is easy to see that the above embedding admits a left adjoint,
which we denote $\on{Compl}$, making $\cO\alg^{\on{nil}}(\cC)$ a localization of
$\cO\alg(\cC)$.

\sssec{}  

It follows from the construction that the essential image of the functor 
$$\cO\alg(\cC)\leftarrow\cO^\vee\nilpcoalg_{\on{d.p.}}(\cC):\coBar^{\rm{enh}}_{\cO^\vee}$$
belongs to $\cO\alg^{\on{nil}}$. 
Let us denote the resulting functor
$$\cO\alg^{\on{nil}}(\cC)\leftarrow\cO^\vee\nilpcoalg_{\on{d.p.}}(\cC)$$
by $\on{KD}_{\cO\leftarrow\cO^\vee}$.

\medskip

By adjunction, we obtain that the functor $\Bar^{\rm{enh}}_\cO$ factors as
$$\cO\alg(\cC)\overset{\on{Compl}} \longrightarrow \cO\alg^{\on{nil}}(\cC)\longrightarrow 
\cO^\vee\nilpcoalg_{\on{d.p.}}(\cC),$$
for a canonically defined functor 
$$\on{KD}_{\cO\to\cO^\vee}:\cO\alg^{\on{nil}}(\cC)\longrightarrow \cO^\vee\nilpcoalg_{\on{d.p.}}(\cC),$$
which is left adjoint to $\on{KD}_{\cO\leftarrow\cO^\vee}$.

\begin{conjecture}  \label{conj:full Koszul}
The adjoint functors
$$\on{KD}_{\cO\to\cO^\vee}:\cO\alg^{\on{nil}}(\cC)\rightleftarrows \cO^\vee\nilpcoalg_{\on{d.p.}}(\cC):
\on{KD}_{\cO\leftarrow\cO^\vee}$$
are equivalences of $\oo$-categories.
\end{conjecture}

\medskip

In the next section we will give a proof of this conjecture in a particular case. 

\begin{remark} The derived notion of Koszul duality discussed here is broadly construed; there is no use made of Koszul resolutions in the sense of \cite{priddy}. It would be equally accurate to call this bar-cobar duality.

\end{remark}

\ssec{Coalgebras over an operad}  \label{ss:all coalgebras}

\sssec{}

Note that the monoidal $\oo$-category $\cX^\Sigma$ of symmetric sequences is endowed with
a different right lax action on $\cC$, i.e., a lax monoidal functor $\cX^\Sigma \ra \Funct(\cC,\cC)$:
\begin{equation} \label{e:action 2}
(\cO=\{\cO(n)\})\squig \left(c\squig \underset{n\geq 1}\Pi\, (\cO(n)\otimes c^{\otimes n})_{\Sigma_n}\right).
\end{equation}

\medskip

Hence, for a cooperad $\cP\in \coOp(\cX)$, it makes sense to talk about $\cP$-comodules in
$\cC$ with respect to this new action. We denote the resulting $\oo$-category of comodules by
$\cP\coalg_{\on{d.p.}}(\cC)$ and call them $\cP$-coalgebras (with divided powers). See \cite{fresse} for a treatment of simplicial $\cO$-algebras with divided powers, where it shown, for instance, that a simplicial Lie algebra with divided powers is a simplicial restricted Lie algebra.

\begin{remark}
Since the above is only a right lax action, a cooperad $\cP$ does not define a comonad
in $\cC$. In particular, the forgetful functor $\cP\coalg_{\on{d.p.}}(\cC)\to \cC$
does not in general admit a right adjoint.
\end{remark}

\medskip

We have an evident forgetful functor
\begin{equation} \label{e:sum vs products one}
\cP\nilpcoalg_{\on{d.p.}}(\cC)\longrightarrow \cP\coalg_{\on{d.p.}}(\cC).
\end{equation}

\begin{remark} One can show that the above functor $\cP\nilpcoalg_{\on{d.p.}}(\cC)\to \cP\coalg_{\on{d.p.}}(\cC)$
is fully faithful and that it admits right adjoint, making the $\oo$-category $\cP\nilpcoalg_{\on{d.p.}}(\cC)$ into a colocalization of
$\cP\coalg_{\on{d.p.}}(\cC)$.
\end{remark}

\sssec{}

Note now that we have yet another action (resp., right lax action) of $\cX^\Sigma$ on $\cC$:
\begin{equation} \label{e:action 3}
(\cO=\{\cO(n)\})\squig \left(c\squig \underset{n\geq 1}\sqcup\, (\cO(n)\otimes c^{\otimes n})^{\Sigma_n}\right),
\end{equation}
and
\begin{equation} \label{e:action 4}
(\cO=\{\cO(n)\})\squig \left(c\squig \underset{n\geq 1}\Pi\, (\cO(n)\otimes c^{\otimes n})^{\Sigma_n}\right)
\end{equation}
respectively.

\medskip

Thus, for a cooperad $\cP\in \coOp(\cX)$ we have two more notions of $\cP$-coalgebras in $\cC$.
We denote the corresponding $\oo$-categories by
$$\cP\nilpcoalg(\cC) \text{ and } \cP\coalg(\cC),$$
respectively. As in the case of divided powers, we have natural forgetful functor
\begin{equation} \label{e:sum vs products two}
\cP\nilpcoalg(\cC)\longrightarrow \cP\coalg(\cC).
\end{equation}

\sssec{}

We also have natural homomorphisms of right lax actions
$$\text{original action }\to \text{ \eqref{e:action 3} }  \text{ and } \text{ \eqref{e:action 2} } \to \text{ \eqref{e:action 4} },$$ 
given by the trace map
$$(-)_{\Sigma_n}\to (-)^{\Sigma_n},$$
(i.e., averaging over the group $\Sigma_n$), 
and the corresponding functors between the $\oo$-categories of comodules:
\begin{equation} \label{e:inv vs coinv}
\cP\nilpcoalg_{\on{d.p.}}(\cC)\to \cP\nilpcoalg(\cC) \text{ and } \cP\coalg_{\on{d.p.}}(\cC)\to \cP\coalg(\cC).
\end{equation}

\medskip

Let us note that when $\cX$ is compatibly tensored over $\Vect_k$, where $k$ has characteristic zero,
the above homomorphisms of actions are homotopy equivalences, and hence the functors 
in \eqref{e:inv vs coinv} are equivalences.

\section{Koszul duality in nilpotent tensor $\oo$-categories}

\ssec{Nilpotent and pro-nilpotent tensor $\oo$-categories}

We retain the setting of Section  \ref{ss:alg and mod}. 

\begin{definition} \label{def:nilp}
We shall say that $\cC$ is pro-nilpotent if it can be exhibited as a limit
$$\cC\simeq \underset{\BN^{\on{op}}}{\lim}\, \cC_i$$
(where the limit is taken in the $(\oo,1)$-category of stable symmetric monoidal $\oo$-categories compatibly tensored over $\cX$),
such that 
\begin{itemize}

\item $\cC_0=0$;

\item For every $i\geq j$, the transition functor $f_{i,j}:\cC_{i}\to \cC_{j}$ commutes with limits;\footnote{It is are also required to commute with colimits, according to our conventions, see Section  \ref{sss:oo-cat}.}

\item For every $i$, the restriction of the tensor product functor $\cC_i\otimes \cC_i\to \cC_i$ to
$\on{ker}(f_{i,i-1})\otimes \cC_i$ is null-homotopic.

\end{itemize}

\smallskip

We shall say that $\cC$ is nilpotent of order $n$, if the functors $f_{i,j}$ are equivalences for 
$i,j\geq n$.
\end{definition}

\medskip

We are going to show: 
\begin{prop} \label{p:duality for nilpotent}
Assume that the operad $\cO$ is such that augmentation map
$\cO(1) \ra 1_\cX$ is a homotopy equivalence. Assume also that $\cC$ is pro-nilpotent. Then
the mutually adjoint functors of \eqref{e:Koszul adj} are 
homotopy equivalences of $\oo$-categories.
\end{prop}

\begin{remark}
The assumption that the map $\cO(1)\ra 1_\cX$ is a homotopy equivalence can be weakened. All we actually need is
that $\cO$ be derived Koszul and that the kernel of $\cO(1)\ra 1_\cX$ act nilpotently on $\cC$.
\end{remark}

\medskip

The rest of this subsection is devoted to the proof of the above proposition.

\sssec{Reduction to the nilpotent case}

Let $\cC$ be written as $\underset{\alpha}{\lim}\, \cC_\alpha$, where the transition functors
commute with limits and colimits. For each index $\alpha$, let $f_\alpha$ denote the
evaluation functor $\cC\to \cC_\alpha$. 

\medskip

The fact that the functors $f_{\alpha,\beta}:\cC_\beta\to \cC_\alpha$
commute with limits (resp., colimits) implies that for every $\alpha$, the functor $f_\alpha$ 
commutes with limits (resp., colimits). I.e., limits (resp., colimits) in $\cC$ can be computed ``component-wise."

\medskip 

We have
$$\cO\alg(\cC)\simeq \underset{\alpha}{\lim}\, \cO\alg(\cC_\alpha),$$
and this equivalence commutes with the corresponding functors $\oblv_\cO$
(this requires no assumption on the transition functors). We also have
$$\cO^\vee\nilpcoalg_{\on{d.p.}}(\cC)\simeq \underset{\alpha}{\lim}\, \cO^\vee\nilpcoalg_{\on{d.p.}}(\cC_\alpha),$$
and this equivalence commutes with the corresponding functors $\oblv_{\cO^\vee}$
(this follows from the above mentioned fact that the functors $f_\alpha$ commute with
colimits).

\medskip

Moreover, we claim that for each $\alpha$, the diagram
$$
\CD
\cO\alg(\cC)  @>{\Bar^{\rm{enh}}_\cO}>>  \cO^\vee\nilpcoalg_{\on{d.p.}}(\cC) \\
@V{e_\alpha}VV   @VV{e_\alpha}V  \\
\cO\alg(\cC_\alpha)  @>{\Bar^{\rm{enh}}_\cO}>>  \cO^\vee\nilpcoalg_{\on{d.p.}}(\cC_\alpha)
\endCD
$$
commutes. This again follows from the fact that the functors $f_\alpha$ commute with colimits.

\medskip

The diagram
$$
\CD
\cO\alg(\cC)  @<{\coBar^{\rm{enh}}_{\cO^\vee}}<<  \cO^\vee\nilpcoalg_{\on{d.p.}}(\cC) \\
@V{e_\alpha}VV   @VV{e_\alpha}V  \\
\cO\alg(\cC_\alpha)  @<{\coBar^{\rm{enh}}_{\cO^\vee}}<<  \cO^\vee\nilpcoalg_{\on{d.p.}}(\cC_\alpha)
\endCD
$$
commutes as well, and this follows from the fact that the functors $f_\alpha$ commute with limits. 

\medskip

The commutativity of the above two diagrams shows that it if the adjoint functors
of \eqref{e:Koszul adj} are equivalences for each $\cC_\alpha$, then they
are also equivalences for $\cC$. 

\sssec{The nilpotence condition}

Thus, from now on we shall assume that $\cC$ is nilpotent. We will use it in the following 
form:

\begin{lemma} \label{l:cobar for nilpotent} 
Assume that $\cC$ is nilpotent.  For any cooperad $\cP$ we have: 

\smallskip

\noindent{\em(a)} The functor $\coBar_{\cP}:\cP\nilpcoalg_{\on{d.p.}}(\cC)\to \cC$
commutes with sifted colimits.

\smallskip

\noindent{\em(b)} The map of monads of Lemma \ref{l:identification of monad}
is a homotopy equivalence.

\end{lemma}

\begin{proof}

We prove point (a): 

By construction, the functor $\coBar_{\cP}$ is the composition of a functor
$$\coBar^\bullet_{\cP}:\cP\nilpcoalg_{\on{d.p.}}(\cC)\to \cC^{\Delta},$$
which commutes with sifted colimits (because the $n$-fold tensor power functor commutes with sifted colimits), followed by the functor
$$\on{Tot}:\cC^{\Delta}\to \cC.$$
Here we denote by $\cC^{\Delta}=\Funct(\Delta,\cC)$ the $\oo$-category of cosimplicial 
objects in $\cC$, and $\on{Tot}$ is the functor of taking the limit over $\Delta$.

\medskip

Let $\cC$ be such that all $n$-fold tensor products are equivalent to zero. This
implies that for $A\in \cP\nilpcoalg_{\on{d.p.}}(\cC)$, the natural map
$$\coBar^\bullet_{\cP}(A)\to \on{cosk}^{\leq n}(\coBar^\bullet_{\cP}(A)|_{\Delta_{\leq n}})$$
is a homotopy equivalence, where $\Delta_{\leq n}\subset \Delta$ is the subcategory spanned
by objects of cardinality $\leq n$. Hence,
$$\coBar_{\cP}(A)\simeq \underset{\Delta_{\leq n}}{\lim}\, \coBar^\bullet_{\cP}(A)|_{\Delta_{\leq n}}.$$

As was mentioned above, the functor 
$$A\squig \coBar^\bullet_{\cP}(A)|_{\Delta_{\leq n}}:\cP\nilpcoalg_{\on{d.p.}}(\cC)\longrightarrow \cC^{\Delta_{\leq n}}$$
commutes with sifted colimits. Hence, the assertion follows from the fact that the functor of limit over 
$\Delta_{\leq n}$
$$\cC^{\Delta_{\leq n}}\to \cC$$
commutes with colimits (since $\cC$ is stable and $\Delta_{\leq n}$ is finite, the limit diagram
is equivalent to a colimit one).

\medskip

To prove point (b), let $\on{Cobar}^\bullet(\cP)$ be the canonical cosimplicial object of $\cX^\Sigma$, such that
$$\on{Tot}(\on{Cobar}^\bullet(\cP))\simeq \on{Cobar}(\cP)=:\cP^\vee.$$ Then $\on{T}_{\cP^\vee}$ is given by
$$c\squig \on{Tot}(\on{Cobar}^\bullet(\cP))\cdot (c),$$
and $\on{Cobar}_{\cP}\circ \on{triv}_{\cP}$ is given by
$$c\squig \on{Tot}(\on{Cobar}^\bullet(\cP)\cdot (c)),$$
where $\on{-}\cdot \on{-}$ denotes the action of $\cX^\Sigma$ on $\cC$. 

\medskip

However, as above, the maps
$$\on{Tot}(\on{Cobar}^\bullet(\cP))\cdot (c) \to \left(\underset{\Delta_{\leq n}}{\lim}\, \on{Cobar}^\bullet(\cP)|_{\Delta_{\leq n}}\right)\cdot c$$
and 
$$\on{Tot}(\on{Cobar}^\bullet(\cP)\cdot (c))\to \underset{\Delta_{\leq n}}{\lim}\, \left(\on{Cobar}^\bullet(\cP)\cdot (c)|_{\Delta_{\leq n}}\right)$$
are homotopy equivalences. Thus, the above totalizations are isomorphic to finite limits, and since $\cC$ is stable, also to colimits.
Therefore, the assertion follows from the fact that the action of $\cX$ on $\cC$ and the monoidal operation on $\cC$ commute with
colimits. 

\end{proof}

Since the functor $\oblv_{\cP{}^\vee}$ is conservative and commutes with colimits,
from point (a) of Lemma \ref{l:cobar for nilpotent} we obtain:

\begin{cor} \label{c:coBar for nilp}
The functor $\coBar^{\rm{enh}}_{\cP}: \cP\nilpcoalg_{\on{d.p.}}(\cC)\to \cP{}^\vee\alg(\cC)$
commutes with geometric realizations.
\end{cor}

\sssec{The functor $\Bar^{\rm{enh}}_\cO$ is fully faithful}

To prove that $\Bar^{\rm{enh}}_\cO$ is fully faithful we need to show that the unit of the adjunction
\begin{equation} \label{e:adj map}
\on{Id}\longrightarrow \coBar^{\rm{enh}}_{\cO^\vee}\circ \Bar^{\rm{enh}}_\cO
\end{equation}
is a homotopy equivalence. 

\medskip

Since every object of $\cO\alg(\cC)$ can be obtained as a geometric realization
of a simplicial object whose terms lie in the essential image of the functor
$$\rm{Free}_\cO:\cC\longrightarrow \cO\alg(\cC),$$
from Corollary \ref{c:coBar for nilp} we obtain that it is enough to show that the map
\begin{equation} \label{e:adj on free}
\rm{Free}_\cO\longrightarrow \coBar^{\rm{enh}}_{\cO^\vee}\circ \Bar^{\rm{enh}}_\cO\circ \rm{Free}_\cO
\end{equation}
is a homotopy equivalence. Again, since the forgetful functor $\oblv_\cO:\cO\alg(\cC)\to \cC$ is
conservative, it is enough to show that the induced map
\begin{equation} \label{e:adj on free and forget}
\oblv_\cO\circ \rm{Free}_\cO\longrightarrow \oblv_\cO\circ \coBar^{\rm{enh}}_{\cO^\vee}\circ \Bar^{\rm{enh}}_\cO\circ \rm{Free}_\cO
\end{equation}
is a homotopy equivalence. 

\medskip

By definition, the functor in the left-hand side of \eqref{e:adj on free and forget} identifies
with $\on{T}_\cO$. We can rewrite the right-hand side of \eqref{e:adj on free and forget} as
$$\coBar_{\cO^\vee}\circ \Bar^{\rm{enh}}_\cO\circ \rm{Free}_\cO.$$
From Diagram \eqref{d:O 3rd}, we obtain a canonical homotopy equivalence of
functors
$$\Bar^{\rm{enh}}_\cO\circ \rm{Free}_\cO\simeq \triv_{\cO^\vee}.$$
Hence, the map in \eqref{e:adj on free and forget} can be thought of as a map
\begin{equation} \label{e:adj on free and forget again}
\on{T}_\cO\longrightarrow \coBar_{\cO^\vee}\circ \triv_{\cO^\vee}.
\end{equation}
However, it is easy to see from the construction that the map in \eqref{e:adj on free and forget again}
equals the composition
$$\on{T}_\cO\longrightarrow \on{T}_{(\cO^\vee)^\vee}\longrightarrow \coBar_{\cO^\vee}\circ \triv_{\cO^\vee},$$
where the first arrow is given by \eqref{e:adj on free and forget}, and the second
arrow is given by Lemma \ref{l:identification of monad}. Hence, the fact that $\cO$ is derived Koszul
(see Remark \ref{r:Koszul for operads}) and Lemma \ref{l:cobar for nilpotent}(b) imply that the map in 
\eqref{e:adj on free and forget again}
is a homotopy equivalence.

\qed

\sssec{Proof of the equivalence}

To prove that the functor $\Bar^{\rm{enh}}_\cO$ is an equivalence, it remains to show
that its right adjoint, namely $\coBar^{\rm{enh}}_{\cO^\vee}$, is conservative. For
that it is sufficient to show that the functor
$$\coBar_{\cP}:\cP\nilpcoalg_{\on{d.p.}}(\cC)\longrightarrow \cC$$
is conservative for a cooperad $\cP$, provided that $\cC$ is nilpotent.

\begin{remark}
Note that Conjecture \ref{conj:full Koszul} would imply that the functor 
$\coBar^{\rm{enh}}_{\cP}$ is conservative for any $\cC$, without the nilpotence
(or pro-nilpotence) assumption.
\end{remark}

\medskip

Let $\cC_i$ be as in Definition \ref{def:nilp}. Let $f_i:\cC\to \cC_i$ denote the
corresponding evaluation functors.

\medskip

Let $\alpha:B_1\to B_2$ be a map in $\cP\nilpcoalg_{\on{d.p.}}(\cC)$
that is not a homotopy equivalence. Let $i$ be the minimal integer such that 
the map
$$f_i(\alpha):f_i(B_1)\to f_i(B_2)$$
is not a homotopy equivalence.

\medskip

For any $B\in \cP\nilpcoalg_{\on{d.p.}}(\cC)$ we have a canonical map
\begin{equation} \label{e:cosk filtr}
\coBar_{\cP}(B)=\underset{\Delta}{\lim}\, \coBar^\bullet_{\cP}(B)\longrightarrow \oblv_\cP(B).
\end{equation}
By the choice of the index $i$, the map 
\begin{equation} \label{e:eval}
f_i\circ \coBar_{\cP}(\alpha):f_i\circ \coBar_{\cP}(B_1)\longrightarrow f_i\circ \coBar_{\cP}(B_2),
\end{equation}
induces a homotopy equivalence
$$\on{coker}\left(\coBar_{\cP}(B_1)\longrightarrow B_1\right)\longrightarrow 
\on{coker}\left(\coBar_{\cP}(B_2)\longrightarrow B_2\right).$$

Hence, the map \eqref{e:eval} is not a homotopy equivalence. Hence $\coBar_{\cP}(\alpha)$
is not a homotopy equivalence, as required.

\qed

\ssec{Coalgebras vs. ind-nilpotent coalgebras in the pro-nilpotent case}  \label{ss:alg and coalg}

In the following, we retain the assumption that $\cC$ is pro-nilpotent:

\begin{prop} \label{p:sum vs prod}
The functors \eqref{e:sum vs products one} and \eqref{e:sum vs products two} are equivalences.
\end{prop}

\begin{proof}

As in the proof of Proposition \ref{p:duality for nilpotent} we immediately reduce to the case
when $\cC$ is nilpotent. In the latter case, in both cases (with or without divided powers), the two 
right lax actions of $\cX^\Sigma$ on $\cC$ are tautologically equivalent by 
the nilpotence condition.

\end{proof}

\ssec{The case of Lie algebras}

\sssec{}

Let $\cX$ be the category $\Vect_k$, where $k$ has characteristic zero. We shall consider the augmented
operad $\Lie$, obtained from the usual (non-unital) Lie operad by formally adjoining the unit. As was
mentioned in Remark \ref{r:Koszul for operads}, the operad $\Lie$ is derived Koszul. 

\sssec{}

Let $\cC$ be a (not necessarily unital) stable symmetric monoidal $\oo$-category, compatibly
tensored over $\Vect_k$. 

\medskip

Let $\Lie\alg(\cC)$ denote the $\oo$-category of Lie algebras in $\cC$, and let $\on{Com}\coalg(\cC)$ denote the
category of non-unital commutative coalgebras on $\cC$. Recall that we have a pair of adjoint functors:
\begin{equation} \label{e:Quillen for Lie general}
\C:\Lie\alg(\cC)\rightleftarrows \on{Com}\coalg(\cC):\prim[-1],
\end{equation}
where $\C$ is the functor of the homological Chevalley complex, and $\prim$ is the (derived) functor of 
primitive elements
(here $[-1]$ stands for the cohomological shift by $1$ to the right, i.e., the loop functor). 

\medskip

We claim that we have proved the following:

\begin{prop} \label{p:Quillen for Lie general}
Assume that $\cC$ is pro-nilpotent. Then the functors in \eqref{e:Quillen for Lie general}
are equivalences of $\oo$-categories.
\end{prop}

\begin{proof}

It is known (see \cite{koszul}, \cite{ching}) that the cooperad $\Lie^\vee$ identifies with $\on{Com}[1]$, i.e., $\on{Com}[1](n)=k[n-1]$
with the sign action of $\Sigma_n$ for every $n$. Moreover, the functor
$$\Bar^{\rm{enh}}_{\Lie}:\Lie\alg(\cC)\longrightarrow \on{Com}[1]\nilpcoalg_{\on{d.p.}}(\cC)\longrightarrow \on{Com}[1]\coalg_{\on{d.p.}}(\cC)\longrightarrow
\on{Com}[1]\coalg(\cC)$$
is the functor $\C[-1]$.

\medskip

From Proposition \ref{p:duality for nilpotent} we obtain an equivalence of $\oo$-categories
$$\Bar^{\rm{enh}}_{\Lie}:\Lie\alg(\cC)\longrightarrow \on{Com}[1]\nilpcoalg_{\on{d.p.}}(\cC).$$

\medskip

By Proposition \ref{p:sum vs prod}, the pro-nilpotence assumption on $\cC$
implies that the functor
$$\on{Com}[1]\nilpcoalg(\cC)\longrightarrow \on{Com}[1]\coalg(\cC)$$
is an equivalence.

\medskip

Due to the characteristic zero assumption, the functor
$$\on{Com}[1]\nilpcoalg_{\on{d.p.}}(\cC)\longrightarrow \on{Com}[1]\nilpcoalg(\cC)$$
is an equivalence as well.

\medskip

Thus, we obtain that $\C\simeq \Bar^{\rm{enh}}_{\Lie}[-1]$ defines an equivalence
$$\Lie\alg(\cC)\longrightarrow \on{Com}\coalg(\cC).$$

\end{proof}

\section{Proof of the main theorem}

\ssec{Koszul duality in the chiral setting}  \label{ss:proof of nilp}

Our current goal is to prove the first part of Theorem \ref{t:main}:

\begin{theorem} \label{t:Koszul}
The above functors 
$$\C^{\ch}:\Lie\alg^{\ch}(\ran X) \rightleftarrows \on{Com}\coalg^{\ch}(\ran X):\prim^{\ch}[-1]$$
are mutually inverse equivalences of $\oo$-categories.
\end{theorem}

In view of Proposition \ref{p:Quillen for Lie general}, it suffices to show that the $\oo$-category
$\fD(\ran X)$, equipped with the chiral symmetric monoidal structure, is pro-nilpotent.

\sssec{}

Let $n$ be a positive integer. For a finite set $I$, let $X^{I,\leq n}$ be the closed subscheme
of $X^I$ equal to the union of the images of the diagonal maps $\Delta(\pi):X^J\to X^I$
for all surjections $\pi:I\twoheadrightarrow J$ with $|J|\leq n$. Let $X^{I,>n}\subset X^I$ be the
complementary open subset. Let
$$\imath^{I,n}:X^{I,\leq n}\hookrightarrow X^I \hookleftarrow X^{I,>n}:\jmath^{I,n}$$
denote the corresponding maps.

\medskip

We obtain the functors
$$X^{\Delta,\leq n},X^{\Delta,> n}:(\fset^{\rm surj})^{\on{op}}\longrightarrow \Sch$$
and the corresponding functors
$$\fD^!(X^{\Delta,\leq n}),\fD^!(X^{\Delta,>n}):\fset^{\rm surj}\longrightarrow \ooCat^{\on{st}}.$$

Let $\fD(\ran^{\leq n} X)$ and $\fD(\ran^{> n} X)$ denote the corresponding $\oo$-categories 
$$\underset{\fset^{\rm surj}}{\lim}\, \fD^!(X^{\Delta,\leq n}) \text{ and }
\underset{\fset^{\rm surj}}{\lim}\, \fD^!(X^{\Delta,> n}),$$
respectively.

\sssec{}

For a surjection $\pi:I_1\twoheadrightarrow I_2$, the map $\Delta(\pi):X^{I_1}\to X^{I_2}$
sends 
$$X^{I_1,\leq n}\longrightarrow X^{I_2,\leq n} \text{ and } X^{I_1,>n}\longrightarrow X^{I_2,>n}.$$
Hence, we obtain commutative diagrams of functors
$$
\CD
\fD(X^{I_1,\leq n})   @<{(\imath^{I_1,n})^! }<<  \fD(X^{I_1,n}) @>{(\jmath^{I_1,n})^*}>>  \fD(X^{I_1,>n})  \\
@V{\Delta(\pi)^!}VV   @V{\Delta(\pi)^!}VV   @VV{\Delta(\pi)^!}V  \\
\fD(X^{I_2,\leq n})   @<{(\imath^{I_2,n})^! }<<  \fD(X^{I_2,n}) @>{(\jmath^{I_2,n})^*}>>  \fD(X^{I_2,>n}) 
\endCD
$$
and their adjoints:
$$
\CD
\fD(X^{I_1,\leq n})   @>{(\imath^{I_1,n})_*}>>  \fD(X^{I_1,n}) @<{(\jmath^{I_1,n})_*}<<  \fD(X^{I_1,>n})  \\
@V{\Delta(\pi)^!}VV   @V{\Delta(\pi)^!}VV   @VV{\Delta(\pi)^!}V  \\
\fD(X^{I_2,\leq n})   @>{(\imath^{I_2,n})_*}>>  \fD(X^{I_2,n}) @<{(\jmath^{I_2,n})_*}<<  \fD(X^{I_2,>n}) 
\endCD
$$

So, we obtain adjoint pairs of functors
$$(\imath^{n})_*:\fD(\ran^{\leq n} X)\rightleftarrows \fD(\ran X):(\imath^{n})^! \text{ and }
(\jmath^{n})^*:\fD(\ran X) \rightleftarrows \fD(\ran^{> n} X):(\jmath_{n})_*$$
that commute with the evaluation maps $(\Delta^I)^!$. Moreover, the functors
$(\imath^{n})_*$ and $(\jmath_{n})_*$ are fully faithful, and 
$$\fD(\ran^{\leq n} X)\rightleftarrows \fD(\ran X)\rightleftarrows \fD(\ran^{> n} X)$$
is a short exact sequence of stable $\oo$-categories: I.e., the category on the
right is the localization of the category in the middle with respect to the category
on the left. 

\medskip

Similarly, we have the corresponding maps and functors for any pair $n_1\leq n_2$.

\begin{lemma}
The functor
$$\{(\imath^n)^!\}:\fD(\ran X)\longrightarrow \underset{n}{\lim}\, \fD(\ran^{\leq n} X)$$
is an equivalence.
\end{lemma}

\begin{proof}
This follows from the fact that each $(\imath^{I,n})^!$ is an equivalence
as soon as $n\geq |I|$.
\end{proof}

\sssec{}

From Lemma \ref{l:ch product on open} we obtain that for any $n$, the essential image
of $\fD(\ran^{> n} X)$ under $(\jmath_{n})_*$ is a monoidal ideal with respect to the chiral
symmetric monoidal structure on $\fD(\ran X)$, i.e., the product of any object with an object in the essential image of $\fD(\ran^{> n} X)$ remains in 
the essential image of $\fD(\ran^{> n} X)$.\footnote{This is a special case of a general notion of ideals of algebras in a pointed monoidal $\oo$-category, 
where a map $I \ra A$ of nonunital algebras is said to be an ideal if the quotient is equivalent to the quotient as objects, without algebraic structure.
For a map $I \ra A$, which is a \emph{monomorphism} (see \cite{topos}, Sect. 5.5.6), this is equivalent to requiring that the resulting maps
$I\otimes A\to A$ and $A\otimes I\to A$ factor through $I$. We are applying this to the monoidal category $\ooCat_{\on{pres}}^{\on{st}}$ and a 
functor $\cC'\to \cC$ 
which is fully faithful, which is equivalent to being a monomorphism.} As a consequence, the localization of $\fD(\ran X)$ with respect to 
$\fD(\ran^{> n} X)$ obtains a monoidal structure, such that the localization functor is a homomorphism of monoidal categories.

\medskip

The localization of $\fD(\ran X)$ with respect $\fD(\ran^{> n} X)$ is equivalent to $\fD(\ran^{\leq n} X)$; hence, we obtain that
$\fD(\ran^{\leq n} X)$ acquires a canonical symmetric monoidal structure, for which
the functors
$$\fD(\ran^{\leq n_1} X)\longrightarrow \fD(\ran^{\leq n_2} X)$$
for $n_1\geq n_2$ are symmetric monoidal for the chiral symmetric monoidal structure on 
$\fD(\ran X)$.

\medskip

To establish the pro-nilpotence property of the chiral symmetric monoidal structure on 
$\fD(\ran X)$ it suffices to show that the resulting monoidal structure on $\fD(\ran^{\leq n} X)$ 
vanishes on
$$\on{ker}\left(\fD(\ran^{\leq n} X)\longrightarrow \fD(\ran^{\leq n-1} X)\right)\otimes \fD(\ran^{\leq n} X)\longrightarrow
\fD(\ran^{\leq n} X).$$

However, the latter is manifest from Lemma \ref{l:ch product on open}.

\ssec{Factorization}  \label{ss:factorization}

We shall now prove the second part of the main theorem, that the equivalence between chiral Lie algebras 
and chiral commutative coalgebras on $\ran X$ interchanges the $\oo$-subcategories of chiral Lie
algebras on $X$ and factorization coalgebras.

\begin{theorem} \label{t:factorization}
For $A\in \Lie\alg^{\ch}(\ran X)$, the corresponding coalgebra $\C^{\ch}(A)$ factorizes if and only
if $A$ is supported on $X$, i.e., is an object of $\Lie\alg^{\ch}(X)$.
\end{theorem}

\medskip

We shall precede the proof of Theorem \ref{t:factorization} by the following two
observations made in Sections \ref{sss:reduction to simple factorization} and
\ref{sss:filtration on Chevalley course}, respectively.

\sssec{}  \label{sss:reduction to simple factorization}

Recall that the factorization condition for $B\in \on{Com}\coalg^{\ch}(\ran X)$ says that 
for each surjection $\pi:I \fib J$ the associated map 
\begin{equation} \label{e:factorization}
\jmath(\pi)^*\left((\Delta^I)^!(B)\right) \longrightarrow 
\jmath(\pi)^* \left(\underset{j\in J}\boxtimes (\Delta^{I_j})^!(B)\right)
\end{equation}
is a homotopy equivalence in $\fD(U(\pi))$.

\medskip

However, we claim that it is enough to check \eqref{e:factorization} for every $I$ and $\pi=\id_I$. 
Indeed, let us assume by induction that the homotopy equivalences \eqref{e:factorization} have been
established for finite sets of cardinality $<k$, and let $I$ be with $|I|=k$. 

\medskip

First, we claim that the induction hypothesis implies that \eqref{e:factorization} becomes
an ismorphism after applying $(\Delta^{\phi})^!\circ \jmath(\pi)_*$ for any
$\phi:I\twoheadrightarrow I'$ with $|I'|<|I|$. Indeed, set $J':=J\underset{I}\sqcup I'$,
and let 
$$\psi:J\twoheadrightarrow J' \text{ and } \phi':I'\twoheadrightarrow J'$$
denote the corresponding maps. We have:
$$(\Delta^{\phi})^!\circ \jmath(\pi)_*\simeq \jmath(\pi')_*\circ (\Delta^{\psi})^!,$$
and thus the situation reduces to that on $X^{I'}$. 

\medskip

Thus, it is sufficient to show that \eqref{e:factorization} becomes
a homotopy equivalence after applying the functor $\jmath(\id_I)^*\circ \jmath(\pi)_*$. However,
in this case the left-hand side becomes $\jmath(\id_I)^*\left((\Delta^I)^!(B)\right)$
while the right-hand side becomes
$$\jmath(\id_I)^*\left(\underset{j\in J}\boxtimes \jmath(\id_{I^j})_*\circ \jmath(\id_{I^j})^*(\Delta^{I_j})^!(B)\right),$$
which by the assumption maps isomorphically to
$$\jmath(\id_I)^*\left(\underset{j\in J}\boxtimes \jmath(\id_{I^j})_*\circ 
\jmath(\id_{I^j})^*\left(\left((\Delta^{\on{main}})^!(B)\right)^{\boxtimes I^j}\right)\right)
\simeq \jmath(\id_I)^*\left(\left((\Delta^{\on{main}})^!(B)\right)^{\boxtimes I}\right).$$
Hence, the map in question becomes the map 
$$\jmath(\id_I)^*\left((\Delta^I)^!(B)\right) \longrightarrow \jmath(\id_I)^*
\left(\left((\Delta^{\on{main}})^!(B)\right)^{\boxtimes I}\right),$$
i.e., the map \eqref{e:factorization} for $\pi=\id_I$.

\sssec{}  \label{sss:filtration on Chevalley course}

The second observation needed for the proof of Theorem \ref{t:factorization} is the
canonical filtration on $\C^{\ch}(A)$ as an object of $\fD(\ran X)$.

\medskip

Let $\cC$ be a stable symmetric monoidal category as in Section  \ref{ss:alg and mod},
and let $L$ be an object of $\Lie\alg(\cC)$. By the construction of the Chevalley complex,
the object $\oblv_{\on{Com}}(\C(L))\in \cC$ carries a canonical filtration indexed by positive
integers with subquotients described as follows:
$$\on{gr}^k\left(\C(L)\right)\simeq \on{Sym}^k_{\cC}(\oblv_{\Lie}(L)[k]).$$

\medskip

We will apply it to $\cC=\fD(\ran X)$ equipped with the chiral symmetric monoidal structure.

\sssec{}

For future use,  let us describe explicitly the object $\on{Sym}^{k,\ch}(M)\in \fD(\ran X)$ for
$M\in \fD(\ran X)$:

\medskip

For a finite set $I$, we have
$$(\Delta^I)^!(\on{Sym}^{k,\ch}(M))\simeq \left(\underset{\pi:I\twoheadrightarrow \{1,...,k\}}\oplus\, 
\jmath(\pi)_*\circ \jmath(\pi)^*\left(\underset{j\in \{1,...,k\}}\boxtimes\, (\Delta^{I_j})^!(M)\right)\right)_{\Sigma_k}.$$

\medskip

Let us consider two particular cases: For $k=1$, we have 
$$(\Delta^I)^!(\on{Sym}^{1,\ch}(M))\simeq (\Delta^I)^!(M).$$

\medskip

Suppose now that $M$ is supported on $X\subset \ran X$, i.e., if it is of the form $M\simeq (\Delta^{\on{main}})_*(M_X)$
for some $M_X\in \fD(X)$. We have that $(\Delta^I)^!(\on{Sym}^{k,\ch}(M))$ is zero unless $|I|=k$, and for
$I$ with $|I|=k$
$$(\Delta^I)^!(\on{Sym}^{k,\ch}(M))\simeq 
\jmath(\id_I)_*\circ \jmath(\id_I)^*(M_X^{\boxtimes I}).$$

\sssec{Proof of Theorem \ref{t:factorization}, the ``if" direction}

Let us first show that if $A$ is supported on $X$, then $\C^{\ch}(A)$ factorizes. By Section  
\ref{sss:reduction to simple factorization}, we need to show that the map 
\begin{equation} \label{e:factorization for ch}
\jmath(\id_I)^*\left((\Delta^I)^!(\C^{\ch}(A))\right) \longrightarrow \jmath(\id_I)^*
\left(\left((\Delta^{\on{main}})^!(\C^{\ch}(A))\right)^{\boxtimes I}\right)
\end{equation}
is a homotopy equivalence.   

\medskip

Let us denote by $A_X$ the D-module on $X$ such that $A=(\Delta^{\on{main}})_*(A_X)$.
Consider the canonical filtration on $\oblv_{\on{Com}}^{\ch}(\C^{\ch}(A))$ of Section  \ref{sss:filtration on Chevalley course}.

We obtain that the functor $\jmath(\id_I)^*$ annihilates all $\on{gr}^k\left((\Delta^I)^!(\C^{\ch}(A))\right)$
except for one with $k=|I|$, and in the latter case we have
$$\jmath(\id_I)^*\left(\on{gr}^k\left((\Delta^I)^!(\C^{\ch}(A))\right)\right)
\simeq \jmath(\id_I)^*(A_X^{\boxtimes I}[|I|]).$$
(In particular, for $k=1$, the map $A_X[1]\longrightarrow(\Delta^{\on{main}})^!(\C^{\ch}(A))$
is a homotopy equivalence.) 

\medskip

Under these identifications, the map of \eqref{e:factorization for ch} becomes
the map
$$\jmath(\id_I)^*\left((\Delta^I)^!(\C^{\ch}(A))\right) \simeq 
\jmath(\id_I)^*\left(\on{gr}^k\left((\Delta^I)^!(\C^{\ch}(A))\right)\right)\simeq 
\jmath(\id_I)^*\left(A_X^{\boxtimes I}[|I|]\right).$$

\qed

\sssec{Proof of Theorem \ref{t:factorization}, the ``only if" direction}

Let us now prove the implication in the opposite direction. Assume that $A\in \Lie\alg^{\ch}(\ran X)$
is such that the underlying D-module is not supported on $X$. Let us show that $\C^{\ch}(A)$ does
not factorize. 

\medskip

By assumption, there exists a finite set $I$ with $|I|\geq 2$,
such that $\jmath(\id_I)^*\left((\Delta^I)^!(A)\right)\neq 0$. Let us take $I$ to be of minimal cardinality among such. 

\medskip

Consider the canonical filtration on $(\Delta^I)^!(\C^{\ch}(A))$, and the induced filtration on 
$$\jmath(\id_I)^*\left((\Delta^I)^!(\C^{\ch}(A))\right).$$ 
By assumption, the only non-vanishing terms of 
$\jmath(\id_I)^*\left(\on{gr}^k\left((\Delta^I)^!(\C^{\ch}(A))\right)\right)$ occur for $k=1$ and $k=|I|$
with the former being canonically isomorphic to $\jmath(\id_I)^*\left((\Delta^I)^!(A)\right)$,
and the latter to
$$\jmath(\id_I)^*\left(\left((\Delta^{\on{main}})^!(\C^{\ch}(A))\right)^{\boxtimes I}\right).$$

\medskip

The map \eqref{e:factorization for ch}, identifies with the map
$$\jmath(\id_I)^*\left((\Delta^I)^!(\C^{\ch}(A))\right)\to 
\jmath(\id_I)^*\left(\on{gr}^k\left((\Delta^I)^!(\C^{\ch}(A))\right)\right),$$
which is not a homotopy equivalence, since it annihilates the first term of the filtration.

\qed

\section{Chiral envelopes of $\star$-Lie algebras}  \label{s:enveloping}

\ssec{The basic commutative diagram}

\sssec{}

By construction, we have a natural map $\otimes^{\star}\to \otimes^{\ch}$ between the two 
symmetric monoidal
structures on $\fD(\ran X)$. More precisely, the identity functor on $\fD(\ran X)$ is a left lax 
symmetric monoidal structure, when viewed as a functor from $\fD(\ran X)$ equipped with the $\star$ 
symmetric monoidal structure to $\fD(\ran X)$ equipped with the chiral monoidal structure.

\medskip

For an operad $\cO$ (resp., cooperad $\cP$) we let $\oblv^{\ch\to \star}_\cO$
(resp., $\oblv_\cP^{\star\to \ch}$) denote the corresponding forgetful functors
$$\cO\alg^{\ch}(\ran X)\to \cO\alg^{\star}(\ran X)  \text{ and }
\cP\coalg^{\star}(\ran X)\to \cP\coalg^{\ch}(\ran X).$$
Both of these functors commute with the forgetful functors to $\fD(\ran X)$. 

\medskip

In particular, we obtain a natural forgetful functor 
\begin{equation} \label{e:forgetful Lie ch to star}
\oblv^{\ch\to \star}_{\Lie}: \Lie\alg^{\ch}(\ran X)\ra \Lie\alg^{\star}(\ran X).
\end{equation}

The above functor is easily seen to commute with limits (since on both sides 
the forgetful functor to $\fD(\ran X)$ is conservative and commutes with limits).
Since the categories involved are presentable, we obtain that the functor in 
\eqref{e:forgetful Lie ch to star} admits a left adjoint.  We denote the resulting
left adjoint functor 
$$\Lie\alg^{\star}(\ran X)\longrightarrow \on{Lie}\alg^{\ch}(\ran X),$$
by $\on{Ind}^{\star\to\ch}_{\Lie}$.

\medskip

Our basic observation is the following:
\begin{prop} \label{p:mantra}
We have a commutative diagram of functors
\begin{equation} \label{diag:mantra}
\CD
\Lie\alg^{\ch}(\ran X)  @>{\C^{\ch}}>>  \on{Com}\coalg^{\ch}(\ran X)  \\
@A{\on{Ind}^{\star\to\ch}_{\Lie}}AA    @AA{\oblv_{\on{Com}}^{\star\to\ch}}A   \\
\Lie\alg^{\star}(\ran X)  @>{\C^{\star}}>>  \on{Com}\coalg^{\star}(\ran X)
\endCD
\end{equation}
\end{prop}

\ssec{Recollections on monads}  \label{ss:monad recollections}

For the proof of Proposition \ref{p:mantra} we need to recall several facts about calculus of monads. 
The general reference for this material is \cite{dag}, Sect. 6.2.

\sssec{}  \label{sss:monad univ ppty}

Recall that for a category $\cC$, a monad
$M$ acting on $\cC$ is, by definition, a unital associative algebra in the monoidal category $\on{Funct}(\cC,\cC)$
of endo-functors on $\cC$. 

\medskip

The monoidal category $\on{Funct}(\cC,\cC)$ acts on $\cC$, so it makes sense to
talk about $M$-modules in $\cC$; we denote this category by $\on{Mod}_M$. 
We shall denote by $\oblv_M$ the forgetful functor $\on{Mod}_M\to \cC$, and by $\ind_M$ its left adjoint.

\medskip

Let $F:\cC\to \cD$ (resp., $G:\cD\to \cC$)
be a functor. There is a natural notion of right (resp., left) action of a monad $M$ on $F$ (resp., $G$):
We view $\on{Funct}(\cC,\cD)$ (resp., $\on{Funct}(\cD,\cC)$) as a right (resp., left) module over
$\on{Funct}(\cC,\cC)$. 

\medskip

If $G$ is the right adjoint of $F$, then the data of action of $M$ on $F$ is equivalent to
that of action of $M$ on $G$. 

\medskip

Moreover, a datum of action of $M$ on $G$ is equivalent to factoring $G$ as a 
composition 
$$\cD\overset{G'}\to \on{Mod}_M\overset{\oblv_M}\longrightarrow \cC.$$
Similarly, a datum of action of $M$ on $F$ is equivalent to factoring $F$ as a composition
$$\cC\overset{\ind_M}\longrightarrow \on{Mod}_M \overset{F'}\to \cD.$$

\sssec{}  \label{sss:best approx}

For an adjoint pair 
$$F:\cC\rightleftarrows \cD:G$$ as above, there exists a universal monad on $\cC$
that acts on $F$ (or, equivalently, on $G$). As a plain endo-functor on $\cC$, this monad is isomorphic
to $G\circ F$. Thus, we can view this construction as endowing $G\circ F$ with a structure of monad. 

\medskip

By the universal property, a datum of action of a monad $M$ on $F$ (resp., $G$) is equivalent to
that of homomorphism of monads $M\to G\circ F$.

\medskip

By Section \ref{sss:monad univ ppty}, the identity map on the monad $G\circ F$ yields a canonical factorization of the functor $G$ as
$$\cD\overset{G^{\on{enh}}}\longrightarrow \on{Mod}_{G\circ F}\overset{\oblv_{G\circ F}}\longrightarrow \cC.$$

\medskip

Thus, we can view the category $\on{Mod}_{G\circ F}$ as ``the best approximation" to $\cD$ from
the point of view of $\cC$. 

\medskip

For the sake of completeness, let us also mention that the Barr-Beck-Lurie
theorem gives a necessary and sufficient condition on the functor $G$, for the resulting
functor $G^{\on{enh}}$ to be an equivalence.

\sssec{}

Let 
$$F:\cC\rightleftarrows \cD:G$$ be as above, and let $M_\cD$ be a monad on $\cD$. We can view
the functor $G\circ M_\cD\circ F$ as the composition of $\ind_{M_\cD}\circ F$ with its right adjoint.
Hence, the above procedure endows $G\circ M_\cD\circ F$ with a structure of monad. 

\medskip

If $M_\cC$ is a monad on $\cC$,
a datum of homomorphism $M_\cC\to G\circ M_\cD\circ F$ is equivalent to a datum of action of
$M_\cC$ on the composition $G\circ \oblv_{M_\cD}$, and hence to that of a commutative
diagram
\begin{equation}  \label{e:G_M}
\CD
\on{Mod}_{M_\cC}  @<{G_M}<<  \on{Mod}_{M_\cD}  \\
@V{\oblv_{M_\cC}}VV    @VV{\oblv_{M_\cD}}V  \\
\cC   @<{G}<<  \cD.
\endCD
\end{equation}
Under such circumstances, we shall denote by $\on{Ind}^F_M$ the left adjoint of $G_M$, which makes
the following diagram commutative:
$$ 
\CD
\on{Mod}_{M_\cC}  @>{\on{Ind}^F_M}>>  \on{Mod}_{M_\cD}  \\
@A{\ind_{M_\cC}}AA    @AA{\ind_{M_\cD}}A  \\
\cC   @>{F}>>  \cD.
\endCD
$$

\sssec{}

The above facts render to the world of comonads by reversing the arrows.

\sssec{}

Let $F:\cC\rightleftarrows \cD:G$, $M_\cC$ and $M_\cD$ be as above. Assume now that both $M_\cC$ and
$M_\cD$ are augmented, and assume that the datum of homomorphism $M_\cC\to G\circ M_\cD\circ F$ is compatible
with the augmentations. This equivalent to extending the diagram \eqref{e:G_M} to a commutative diagram
\begin{equation} \label{e:extended diagram monads}
\CD
\cC   @<{G}<<  \cD \\
@V{\on{triv}_{M_\cC}}VV   @VV{\on{triv}_{M_\cD}}V  \\
\on{Mod}_{M_\cC}  @<{G_M}<<  \on{Mod}_{M_\cD}  \\
@V{\oblv_{M_\cC}}VV    @VV{\oblv_{M_\cD}}V  \\
\cC   @<{G}<<  \cD,
\endCD
\end{equation}
where $\on{triv}_{M_\cC}$ (resp., $\on{triv}_{M_\cD}$) is the functor corresponding to the augmentation on
$M_\cC$ (resp., $M_\cD$).

\medskip

Let $N_\cC$ be the Koszul dual comonad, i.e., the one corresponding to the adjoint pair of functors
$$\Bar_{M_\cC}:\on{Mod}_{M_\cC}\rightleftarrows \cC:\on{triv}_{M_\cC}.$$
By Section \ref{sss:best approx}, the functor $\Bar_{M_\cC}$ canonically factors
as 
$$\on{Mod}_{M_\cC} \overset{\Bar^{\on{enh}}_{M_\cC}}\longrightarrow \on{Comod}_{N_\cC}\overset{\oblv_{N_\cC}}
\longrightarrow \cC,$$
and similarly for the monad $M_\cD$ acting on $\cD$.

\medskip

We claim that we have a natural homomorphism of comonads $F\circ N_\cC\circ G \to N_\cD$. Indeed, defining
such homomorphism is equivalent to making the comonad $N_\cD$ coact on the functor
$\on{triv}_{M_\cC}\circ G$. However, the latter functor is isomorphic to $G_M\circ \on{triv}_{M_\cD}$,
and $\on{triv}_{M_\cD}$ is canonically coacted on by $N_\cD$. 

\medskip

Thus, we obtain a commutative diagram of functors
$$
\CD
\on{Comod}_{N_\cC}  @>{F_N}>> \on{Comod}_{N_\cD}   \\
@V{\oblv_{N_{\cC}}}VV    @VV{\oblv_{N_{\cD}}}V  \\
\cC  @>{F}>>  \cD.
\endCD
$$

In the above circumstances, we claim:

\begin{lemma} \label{l:mantra general}
The following diagram of functors canonically commutes:
$$
\CD
\on{Mod}_{M_\cC} @>{\on{Ind}^F_M}>>   \on{Mod}_{M_\cD} \\
@V{\Bar^{\on{enh}}_{M_\cC}}VV    @VV{\Bar^{\on{enh}}_{M_\cD}}V  \\  
\on{Comod}_{N_\cC}   @>{F_N}>>  \on{Comod}_{N_\cD}.
\endCD
$$
\end{lemma}

\begin{proof}
The diagram
\begin{equation} \label{e:reduced diagram}
\CD
\on{Mod}_{M_\cC} @>{\on{Ind}^F_M}>>   \on{Mod}_{M_\cD} \\
@V{\Bar_{M_\cC}}VV    @VV{\Bar_{M_\cD}}V  \\  
\cC   @>{F}>>  \cD
\endCD
\end{equation} 
naturally commutes, being obtained from the top square in \eqref{e:extended diagram monads}, i.e.,
\begin{equation} \label{e:reduced diagram right}
\CD
\on{Mod}_{M_\cC} @<{G_M}<<   \on{Mod}_{M_\cD} \\
@A{\on{triv}_{M_\cC}}AA   @AA{\on{triv}_{M_\cD}}A  \\
\cC   @<<{G}<  \cD.
\endCD
\end{equation}
by taking the left adjoints. 

\medskip

Thus, we need to show that the two coactions of $N_\cD$ on the resulting functor 
$$\on{Mod}_{M_\cC} \to \cD$$
corresponding to the two circuits in the diagram \eqref{e:reduced diagram} are homotopy equivalent.
This is, in turn, equivalent to showing that the the two coactions on the composed functor
$\cD\to \on{Mod}_{M_\cC}$ in \eqref{e:reduced diagram right} are homotopy equivalent. 
However, the latter follows from the construction
of the homomorphism of comonads $F\circ N_\cC\circ G \to N_\cD$.

\end{proof}

\sssec{Proof of Proposition \ref{p:mantra}}

To prove Proposition \ref{p:mantra}, we apply Lemma \ref{l:mantra general} to $\cC=\cD=\fD(\ran X)$
with $M_\cC$ being the monad $\on{T}^\star_{\Lie}$ and $M_\cD$ being the monad $\on{T}^{\ch}_{\Lie}$,
and $F$ being the identity functor.

\qed

\ssec{Chiral homology of chiral envelopes} 

\sssec{}  \label{sss:direct image}

Let $f:X\to Y$ be a map of schemes. The presentation of $\fD(\ran X)$ as in \eqref{e:Ran as colimit}
defines a functor
$$(f^{\ran})_*:\fD(\ran X)\longrightarrow \fD(\ran Y),$$
via $(f^I)_*:\fD(X^I)\to \fD(Y^I)$ for $I\in \fset^{\rm surj}$. The next lemma results from the definitions:

\begin{lemma}  \label{l:chiral homology}
The functor $(f^{\ran})_*$ has a natural symmetric monoidal functor with respect to the $\star$ symmetric
monoidal structure on $\fD(\ran X)$ and $\fD(\ran Y)$.
\end{lemma} 

\sssec{}

Let us take in the previous setup $Y=\on{pt}$.
We shall denote the resulting symmetric monoidal functor 
$\fD(\ran X)\to \Vect_k$ by 
$$\Gamma_{\rm DR}\left(\ran X,-\right).$$

Being symmetric monoidal, this functor gives rise to a functor
$$\Gamma_{\rm DR}\left(\ran X,-\right)_\cO:\cO\alg(\fD(\ran X))\to \cO\alg(\Vect_k)$$
for any operad $\cO$ and
$$\Gamma_{\rm DR}\left(\ran X,-\right)_\cP:\cP\coalg(\fD(\ran X))\to \cP\coalg(\Vect_k)$$
for any cooperad $\cP$.

\sssec{}

Let us recall from \cite{bd}, Sect. 4.2, that the functor of \emph{chiral homology}
$$\underset{X}\int:\Lie\alg^{\ch}(\ran X)\to \Vect_k$$
is by definition the composition
$$\Lie\alg^{\ch}(\ran X)\overset{\C^{\ch}}\longrightarrow \on{Com}\coalg^{\ch}(\ran X)
\overset{\oblv^{\ch}_{\on{Com}}}\longrightarrow \fD(\ran X)\overset{\Gamma_{\rm DR}\left(\ran X,-\right)}
\longrightarrow \Vect_k.$$

\sssec{}

We shall now prove the following:

\begin{prop}  \label{p:envelopes} The following diagram of functors
$$
\CD
\Lie\alg^\star(\ran X)  @>{\on{Ind}^{\star\to\ch}_{\Lie}}>>  \Lie\alg^{\ch}(\ran X) \\
@V{\Gamma_{\rm DR}(\ran X,)_{\Lie}}VV   @VV{\underset{X}\int}V  \\
\Lie\alg(\Vect_k)   @>{\oblv_{\on{Com}}\circ \C}>>  \Vect_k
\endCD
$$
is canonically commutative.
\end{prop}

\begin{proof}

First, applying Proposition \ref{p:mantra}, we rewrite the composition
$$\Lie\alg^\star(\ran X)\overset{\on{Ind}^{\star\to\ch}_{\Lie}}\longrightarrow \Lie\alg^{\ch}(\ran X) 
\overset{\underset{X}\int}\longrightarrow \Vect_k$$
as 
$$\Lie\alg^{\star}(\ran X)\overset{\C^{\star}}\longrightarrow \on{Com}\coalg^{\star}(\fD(\ran X))
\overset{\oblv^{\star}_{\on{Com}}}\longrightarrow \fD(\ran X)\overset{\Gamma_{\rm DR}\left(\ran X,-\right)}\to \Vect_k,$$
and further as 
$$\Lie\alg^{\star}(\ran X)\overset{\C^{\star}}\longrightarrow \on{Com}\coalg^{\star}(\fD(\ran X))
\overset{\Gamma_{\rm DR}\left(\ran X,-\right)_{\on{Com}}}\longrightarrow 
\on{Com}\coalg(\Vect_k)\overset{\oblv_{\on{Com}}}\longrightarrow \Vect_k.$$

\medskip

Hence, to prove the proposition, it suffices to show that the following diagram of functors
is commutative:
$$
\CD
\Lie\alg^{\star}(\ran X)   @>{\C^{\star}}>>  \on{Com}\coalg^{\star}(\fD(\ran X)) \\  
@V{\Gamma_{\rm DR}(\ran X,-)_{\Lie}}VV       @VV{\Gamma_{\rm DR}(\ran X,-)_{\on{Com}}}V   \\
\Lie\alg(\Vect_k)  @>{\C}>>  \on{Com}\coalg(\Vect_k).
\endCD
$$
However, this follows from Lemma \ref{l:mantra general}: 

\medskip

We apply this lemma it to 
$\cC=\fD(\ran X)$, $\cD=\Vect_k$, $M_\cC=\on{T}_{\Lie}^\star$, 
$M_\cD=\on{T}_{\Lie}$, and $F=\Gamma_{\rm DR}\left(\ran X,-\right)$.
Note that the functor $\on{Ind}^F_M$ of Lemma \ref{l:mantra general} is isomorphic in our
case to just $\Gamma_{\rm DR}(\ran X,-)_{\Lie}$, since $\Gamma_{\rm DR}(\ran X,-)$ is monoidal
and not just left lax monoidal.

\end{proof}

\ssec{Chiral envelopes and factorization}

\sssec{}

Our current goal is to prove the following:

\begin{theorem}  \label{t:actual chiral envelope}
The functor $\on{Ind}^{\star\to\ch}_{\Lie}$ sends the subcategory 
$$\Lie\alg^\star(X)\subset \Lie\alg^\star(\ran X)$$ to the subcategory
$$\Lie\alg^{\ch}(X)\subset \Lie\alg^{\ch}(\ran X).$$
\end{theorem}

Before we prove Theorem \ref{t:actual chiral envelope}, let us derive some corollaries:

\begin{cor}  \label{c:actual chiral envelope}
The resulting functor 
\begin{equation} \label{e:actual chiral envelope}
\on{Ind}^{\star\to\ch}_{\Lie}:\Lie\alg^\star(X)\to \Lie\alg^{\ch}(X)
\end{equation}
is the left adjoint of the forgetful functor
$$\Lie\alg^\star(X)\leftarrow \Lie\alg^{\ch}(X):\oblv^{\ch\to \star}_{\Lie}.$$
\end{cor}

\medskip

We shall sometimes use the notation $U^{\ch}$ for the functor in 
\eqref{e:actual chiral envelope}. This is the higher-dimensional
and derived version of the chiral enveloping functor of \cite{bd}, Sect. 3.7.

\medskip

From Proposition \ref{p:envelopes} we obtain:
\begin{cor} \label{c:envelopes}
For $L\in \Lie\alg^\star(X)$ there exists a canonical homotopy equivalence
$$\underset{X}\int\, U^{\ch}(L) \simeq \oblv_{\on{Com}}\circ \C\left(\Gamma_{\rm DR}(\ran X,L)_{\Lie}\right).$$
\end{cor}

\begin{remark} 
In the situation of the above corollary, let $L_X$ be the D-module on $X$, such that
$$(\Delta^{\on{main}})_*(L_X)\simeq \oblv^\star_{\Lie}(L).$$
Note that 
$$\oblv_{\Lie}\left(\Gamma_{\rm DR}(\ran X,L)_{\Lie}\right)\simeq \Gamma_{\rm DR}(X,L_X),$$
which gives the object $\Gamma_{\rm DR}(X,L_X)\in \Vect_k$ a canonical Lie algebra structure.
Thus, Corollary \ref{c:envelopes} gives a conceptual proof of (a generalization of) a theorem from \cite{bd}, Sect. 4.8.1
that computes the chiral homology of chiral envelopes of $\star$-Lie algebras.
\end{remark}

\begin{remark}  \label{r:unital chiral homology}
The actual theorem of \cite{bd} is slightly different from ours. Namely, in {\it loc.cit.}
one considers the unital version of $U^{\ch}(L)$, and proves the result about its
chiral homology. Thus, in order to obtain their formulation one needs to complement
Corollary \ref{c:envelopes} by one more theorem that shows
that chiral homology of a non-unital chiral Lie algebra $A$ differs from the chiral homology of the corresponding unital 
chiral Lie algebra by a copy of the ground field $k$,
provided that $X$ is connected; see {\it loc.cit.}, Proposition 4.4.8. 
\end{remark}

\begin{remark}

Note that Theorem \ref{t:actual chiral envelope} allows to construct non-commutative chiral Lie algebras
on $X$, for $X$ of any dimension: start with a $\star$-Lie algebra $L$ and take $U^{\ch}(L)$.

\medskip

For example, let $L'$ be a Lie algebra in $\Vect_k$.
Then the D-module $L:=\underline{L'}$ on $X$, corresponding to the ``constant sheaf" with
fiber $L'$, is naturally a $\star$-Lie algebra on $X$. 
Thus, for any $L'$ as above, we can produce the chiral Lie algebra $U^{\ch}(\underline{L'})$.

\medskip

As another example, we can take $L=L'\otimes D_X$, where $D_X\in \fD(X)$ is the
ring of differential operators. The structure of $\star$-Lie algebra on $L'\otimes D_X$
is defined as in \cite{bd}, Example 2.5.6(b)(ii). Or we can consider $L=\Theta_X\underset{\cO_X}\otimes D_X$,
where $\Theta_X$ is the algebroid of vector fields on $X$, see \cite{bd}, Example 2.5.6(b)(i).\footnote{The question of constructing central extensions of these examples \`a la Kac-Moody or Virasoro
is much more subtle.}

\medskip

Note, however, that by Remark \ref{r:pbw}, unless $\dim(X)=1$, if we start with $L$ which lies in the 
heart of the natural t-structure on $X$ and is flat as a quasi-coherent sheaf, the chiral Lie algebra
$U^{\ch}(L)$ considered as a D-module on $X$, will not lie in the heart of the t-structure. This is closely analogous to the topological setting: For $n\geq 2$, any $\cE_n$-algebra over a field of characteristic zero that lies in the heart of the t-structure on chain complexes (i.e., is discrete) has a commutative algebra structure.
\medskip

By the same remark, if we want to obtain $U^{\ch}(L)$, which up to a cohomological shift,
lies in the heart of the t-structure, we typically need to start with $L$, such that $L[1-\dim(X)]$ lies
in the heart of the t-structure. However, the $\star$-Lie algebra structure on such $L$ is automatically trivial, unless $\dim(X)=1$. Likewise, in the topological setting, the $\cE_n$-enveloping algebra of a Lie algebra never lies in the heart of the t-structure, for $n\geq 2$.

\medskip

To summarize: In higher dimensions, it is difficult to produce non-commutative chiral Lie algebras
that lie in the heart of the t-structure on $\fD(X)$.

\end{remark}

\sssec{}  \label{sss:filtration on Chevalley}

For the proof of Theorem \ref{t:actual chiral envelope} let us recall the setting of 
Section  \ref{sss:filtration on Chevalley course}. We shall need one more property
of this construction, which we shall state in a form which is somewhat crude, but will suffice 
for our purposes.

\medskip

Let $\cC$ and $L$ be as in Section  \ref{sss:filtration on Chevalley course}. For a positive
integer $k$ let $\C(L)_{\leq k}$ denote the corresponding term of the filtration on $\oblv_{\on{Com}}(\C(L))$. 
We claim that the coalgebra structure on $\C(L)$ is compatible with the filtration in the 
following weak sense:

\medskip

\noindent For positive integers $k$ and $n$ and a partition $k=k_1+\cdots+k_n$
we have a map
$$\C(L)_{\leq k}\to \C(L)_{\leq k_1}\otimes...\otimes \C(L)_{\leq k_n},$$
satisfying the natural associativity property. 
For $k'\geq k$ and $k'_i\geq k_i$, $i=1,\ldots,n$ the diagram
$$
\CD
\C(L)_{\leq k} @>>>  \C(L)_{\leq k_1}\otimes...\otimes \C(L)_{\leq k_n}\\
@VVV    @VVV  \\
\C(L)_{\leq k'} @>>>  \C(L)_{\leq k'_1}\otimes...\otimes \C(L)_{\leq k'_n}
\endCD
$$
is commutative, and the diagram
$$
\CD
\C(L)_{\leq k} @>>>  \C(L)_{\leq k_1}\otimes...\otimes \C(L)_{\leq k_n} \\
@VVV    @VVV  \\
\oblv_{\on{Com}}(\C(L)) @>>>  \oblv_{\on{Com}}(\C(L))\otimes...\otimes \oblv_{\on{Com}}(\C(L))
\endCD
$$
is commutative as well.  In particular, for $k$ and $n$ as above, 
we obtain the maps
\begin{equation} \label{e:comult on gr}
\on{gr}^k(\C(L))\to \on{gr}^{k_1}(\C(L))\otimes\ldots \otimes \on{gr}^{k_n}(\C(L)),
\end{equation}
that also have the natural associativity property.

\medskip

The final property that we need is the following: 

\medskip

Recall the identification 
$$\on{gr}^k(\C(L))\simeq \on{Sym}^k_\cC(L)\simeq \on{gr}^k(\C(L_{\triv})),$$
where $L_{\triv}:=\triv_{\Lie}\circ \oblv_{\Lie}(L)$.
We obtain that the diagrams
\begin{equation} \label{e:comult on gr and triv}
\CD
\on{gr}^k(\C(L))  @>>>  \on{gr}^{k_1}(\C(L))\otimes\ldots \otimes \on{gr}^{k_n}(\C(L)) \\
@AAA    @AAA  \\
\on{gr}^k(\C(L_{\triv}))  @>>>  \on{gr}^{k_1}(\C(L_{\triv}))\otimes\ldots \otimes \on{gr}^{k_n}(\C(L_{\triv}))
\endCD
\end{equation}
commute, where in the upper horizontal row we use the map \eqref{e:comult on gr}, and in the
lower horizontal row the map is \eqref{e:comult on gr} for $L_{\triv}$. 

\sssec{Proof of Theoem \ref{t:actual chiral envelope}}

By Theorem \ref{t:factorization}, it suffices to show that for $L\in \Lie\alg^\star(\ran X)$
$$B:=\C^{\ch}(\on{Ind}^{\star\to\ch}_{\Lie}(L))\in \on{Com}\coalg^{\ch}(\ran X)$$ factorizes.
We will use the discussion in Section  \ref{sss:reduction to simple factorization} and show that for every
finite set $I$ the corresponding map
\begin{equation} \label{e:factorization map again}
\jmath(\id_I)^*\left((\Delta^I)^!(B)\right)\to \jmath(\id_I)^*\left(\left((\Delta^{\on{main}})^!(B)\right)^{\boxtimes I}\right)
\end{equation}
is a homotopy equivalence in $\fD(U(\id_I))$. 

\medskip

By Proposition \ref{p:mantra}, we have:
$$B\simeq \oblv_{\on{Com}}^{\star\to\ch}(\C^\star(L)).$$

Consider now the filtration on both sides of \eqref{e:factorization map again} given by the filtration on
$$\oblv_{\on{Com}}^\star(\C^\star(L))$$
as in Section  \ref{sss:filtration on Chevalley}.

\medskip

We obtain that it is sufficient to show that the maps
$$\on{gr}^\bullet(B)\to \underset{I}\otimes^{\star}\,  \on{gr}^\bullet(B)$$
of \eqref{e:comult on gr} become homotopy equivalences after applying 
$\jmath(\id_I)^*\circ (\Delta^I)^!$. However, \eqref{e:comult on gr and triv}
allows to reduce the latter assertion to the case when $L$ has the trivial
Lie algebra structure.

\medskip

Thus, we have to show that for $M\in \fD(\ran X)$ of the form $(\Delta^{\on{main}})_*(M_X)$,
and the coalgebra 
$$B:=\on{Sym}^{\bullet,\star}(M),$$
the maps \eqref{e:factorization map again} are homotopy equivalences.

\medskip

However, it is easy to see that
$$\jmath(\id_I)^*\circ (\Delta^I)^!\left(\on{Sym}^{n,\star}(M)\right)\simeq
\underset{n=\underset{i\in I}\Sigma\, n_i}\oplus\, \jmath(\id_I)^*\left(\underset{i}\boxtimes\, \on{Sym}^{n_i,!}(M_X)\right),$$
where $\on{Sym}^{n_i,!}(M_X)$ denotes the corresponding symmetric power taken in category $\fD(X)$,
with respect to the symmetric monoidal structure given by tensor product (see Section \ref{sss:! product}).
This makes the homotopy equivalence \eqref{e:factorization map again} for $\on{Sym}^{\bullet,\star}(M)$ manifest.

\qed

\ssec{The Poincar\'e-Birkhoff-Witt theorem}

\sssec{}

We shall now use Theorem \ref{t:actual chiral envelope} and Proposition \ref{p:mantra} to prove
a generalized version of the PBW theorem of chiral universal enveloping algebras, stated in the original 
form as Theorem 3.7.14 of \cite{bd}.

\medskip

Thus, let $L$ be a $\star$-Lie algebra on $X$, and let $U^{\ch}(L)\in \Lie\alg^{\ch}(X)$ be its 
chiral universal enveloping algebra. Let $U^{\ch}(L)_X$ denote the corresponding object of
$\fD(X)$.

\medskip

By Theorem \ref{t:actual chiral envelope} and Proposition \ref{p:mantra}, we have a homotopy equivalence of
D-modules on $X$:
\[U^{\ch}(L)_X\simeq (\Delta^{\on{main}})^!\left(\oblv^{\star}_{\on{Com}}(\C^{\star}(L))\right)[-1].\]

The filtration of Section  \ref{sss:filtration on Chevalley course} on 
$\oblv^{\star}_{\on{Com}}(\C^{\star}(L))$ defines a filtration on $U^{\ch}(L)_X$.

\begin{cor} \label{c:PBW}
The associated graded $\on{gr}^\bullet(U^{\ch}(L)_X)$ is canonically isomorphic to
$$\on{Sym}^{\bullet,!}(L_X[1])[-1],$$
where $(\Delta^{\on{main}})_*(L_X)\simeq L$.
\end{cor}

\sssec{Proof of Corollary \ref{c:PBW}}

The proof follows immediately from the homotopy equivalences
$$\on{gr}^\bullet\left(\oblv^{\star}_{\on{Com}}(\C^{\star}(L))\right)\simeq \on{Sym}^{\bullet,\star}(L[1]),$$
and for $M=(\Delta^{\on{main}})_*(M_X)$,
$$(\Delta^{\on{main}})^!\left(\on{Sym}^{\bullet,\star}(M)\right)\simeq \on{Sym}^{\bullet,!}(M_X).$$ 

\qed

\begin{remark}  \label{r:pbw}
Assume that $X$ is smooth of dimension $d$, and $L$ is such that $L[1-d]$ lies in the heart of the natural
t-structure on $X$, and is flat as a quasi-coherent sheaf. Then Corollary \ref{c:PBW} implies that
$U^{\ch}(L)_X[1-d]$ also lies in the heart of the t-structure.
\end{remark}

\ssec{$\star$-Factorization coalgebras}

\sssec{}

Let us denote by $\Fact^\star(X)$ the full subcategory of $\on{Com}\coalg^\star(\ran X)$ equal to the preimage under
$$\oblv_{\on{Com}}^{\star\to\ch}:\on{Com}\coalg^\star(\ran X)\to \on{Com}\coalg^{\ch}(\ran X)$$
of the full subcategory $\Fact(X)\subset \on{Com}\coalg^{\ch}(\ran X)$.

\medskip

We can encode Theorems \ref{t:main} and \ref{t:actual chiral envelope} and Proposition \ref{p:mantra}
as the following commutative diagram:

\[\xymatrix{
\Lie\alg^{\ch}(X)
\ar@/^1.5pc/[rrrr]^{\sim}
\ar@{_{(}->}[r]
&
\Lie\alg^{\ch}(\ran X)
\ar[rr]^{\sim}_{\C^{\ch}}
&&
\on{Com}\coalg^{\ch}(\ran X)
&
\Fact(X)\ar@{^{(}->}[l]\\
\Lie\alg^\star(X)
\ar@/^1.5pc/[rrrr]
\ar[u]^{U^{\ch}}
\ar@{_{(}->}[r]
&
\Lie\alg^\star(\ran X)
\ar[u]^{\on{Ind}^{\star\to\ch}_{\Lie}}
\ar[rr]_{\C^\star}
&&
\on{Com}\coalg^{\star}(\ran X)
\ar[u]_{\oblv_{\on{Com}}^{\star\to\ch}} 
&
\Fact^\star(X)
\ar[u]_{\oblv_{\on{Com}}^{\star\to\ch}}
\ar@{^{(}->}[l]\\
}\]

\sssec{}

Note that, unlike $\C^{\ch}$, the functor
$$\C^{\star}:\Lie\alg^\star(\ran X)\to \on{Com}\coalg^\star(\ran X)$$
is not an equivalence, since the category $\fD(\ran X)$ equipped with the
$\star$ symmetric monoidal functor is not pro-nilpotent. For instance, for $X=(\on{pt}):=\Spec k$, the above functor is the usual
functor
$$\C:\Lie\alg(\Vect_k)\to \on{Com}\coalg(\Vect_k),$$
which is not an equivalence (since we include no nilpotence hypotheses on the algebras). 

\medskip

This example embeds into the case of any $X$ by choosing a $k$-point $x\in X$, and thus realizing
$\Vect_k\simeq \fD(\ran (\on{pt}))$ as a full subcategory of $\fD(\ran X)$.

\section{Chiral and factorization modules}  \label{s:modules}

\ssec{Modules for algebras over an operad}

\sssec{}

We return to the setting of Section  \ref{ss:alg and mod}. Let $\cM$
be a module $\oo$-category for $\cC$. I.e., $\cM$ is a $\cC$-module in the
symmetric monoidal $(\oo,1)$-category of $\cX$-modules in 
$\ooCat^{\on{st}}_{\on{pres},\on{cont}}$. 

\medskip

We can consider $\cM\times \cC$ as another symmetric monoidal $\oo$-category,
where the monoidal operation on  $\cM\times 0_\cC$ is zero. Let 
$$p:\cM\times \cC\longrightarrow \cC,\,\, (m\times c)\squig c$$
denote the tautological homomorphism. 

\begin{definition}
The $\oo$-category $\m_A(\cM)$ is the fiber of the functor
$$p:\cO\alg(\cM\times \cC)\longrightarrow \cO\alg(\cC)$$
over $A$.
\end{definition}

The natural forgetful functor
$$\oblv_A:\m_A(\cM)\longrightarrow \cM\times \cC\longrightarrow \cM$$
admits a left adjoint, denoted $\Free_A$. The composition
$$\on{T}_A:=\Free_A\circ \oblv_A:\cM\longrightarrow \cM$$
is naturally a monad on $\cM$, and by the Barr-Beck-Lurie theorem
$\m_A(\cM)\simeq \m_{\on{T}_A}(\cM)$.

\medskip

Similarly, for a cooperad $\cP$ and $B\in \cP\nilpcoalg_{\on{d.p.}}(\cC)$,
we introduce an $\oo$-category $$\Comod^{\rm nil}_B(\cM),$$ endowed with a forgetful
functor $\oblv_B:\Comod^{\rm nil}_B(\cM)\to \cM$, which admits a right adjoint
$$\coFree_B:\cM\to \Comod^{\rm nil}_B(\cM),$$ so that
$$\Comod^{\rm nil}_B(\cM)\simeq \Comod_{\on{S}_B}(\cM),$$
where 
$\on{S}_{B}:=\oblv_B\circ \coFree_B$.

\medskip

It is easy to see from the construction that the $\oo$-categories $\m_A(\cM)$ and $\Comod^{\rm nil}_B(\cM)$ are stable.

\sssec{}

Let $\cO$ and $A$ be as above. Set $\cO^\vee:=\Bar(\cO)$ and $A^\vee:=\Bar^{\on{enh}}_\cO(A)$.
We have a tautological functor
$$\triv_A:\cM\longrightarrow \m_A(\cM),$$
which commutes with limits and colimits. We denote by 
$$\Bar_A:\m_A(\cM)\longrightarrow \cM$$
the left adjoint of $\triv_A$.

\begin{lemma}
The comonad $\Bar_A\circ \triv_A:\cM\to \cM$ is canonically equivalent to $\on{S}_{A^\vee}$.
\end{lemma}

Hence, the functor $\Bar_A$ canonically upgrades to a functor
\begin{equation} \label{e:KD for modules}
\Bar^{\on{enh}}_A:\m_A(\cM)\longrightarrow \Comod^{\rm nil}_{A^\vee}(\cM),
\end{equation}
such that 
$$\Bar_A\simeq \oblv_{A^\vee}\circ \Bar^{\on{enh}}_A,$$
where $\oblv_{A^\vee}:\Comod^{\rm nil}_{A^\vee}(\cM)\to \cM$ is the forgetful functor.

\medskip

\begin{definition} \label{def:pro-nilp I}
We shall say that a $\cC$-module $\cM$ is pro-nilpotent if $\cM$
can be exhibited as 
$$\cM\simeq \underset{\BN^{\on{op}}}{\lim}\, \cM_i$$
(where the limit is taken in the $(\oo,1)$-category of $\cC$-modules),
such that 
\begin{itemize}

\item $\cM_0=0$;

\item For every $i\geq j$, the transition functor $f_{i,j}:\cM_{i}\to \cM_{j}$ commutes with limits;

\item For every $i$, the restriction of the action functor $\cC\otimes \cM_i\to \cM_i$ to
$\cC\otimes \on{ker}(f_{i,i-1})$ is null-homotopic.

\end{itemize}

\end{definition}

\medskip

As in Proposition \ref{p:duality for nilpotent} one proves:

\begin{prop}  \label{p:Koszul duality for modules}
Let $\cO\in \Op(\cX)$ be an operad, and $A\in \cO\alg(\cC)$ an $\cO$-algebra in $\cC$,
such that the adjunction $A\to \coBar^{\cO^\vee}(A^\vee)$ is a homotopy equivalence. Then
for a pro-nilpotent $\cC$-module $\cM$, the functor \eqref{e:KD for modules} 
$$\Bar^{\on{enh}}_A:\m_A(\cM)\longrightarrow \Comod^{\rm nil}_{A^\vee}(\cM)$$
is an equivalence.
\end{prop}

\sssec{}

Let us take $\cX=\Vect_k$, where $\on{char}(k)=0$, and $\cO=\on{\Lie}$. From
Proposition \ref{p:Koszul duality for modules} and considerations analogous to those in
Section  \ref{ss:alg and coalg} we obtain:

\begin{cor} \label{c:KD Lie for modules}
Let $\cC$ be pro-nilpotent, and let $\cM$ be a pro-nilpotent $\cC$-module. Then for
$L\in \on{Lie}\alg(\cC)$ and $B:=L^\vee\in\on{Com}\coalg(\cC)$, the homological Chevalley
complex functor
$$\C_L:\m_L(\cM)\longrightarrow \Comod_{B}(\cM)$$
is an equivalence of $\oo$-categories.
\end{cor}

\ssec{Chiral and $\star$ actions}

Let us take $\cC$ to be $\fD(\ran X)$ with either the $\star$ or the chiral symmetric monoidal
structures. We shall now recall a class of $\fD(\ran X)$-module $\oo$-categories for both
structures. These categories we first introduced in \cite{ro}. 

\sssec{}

Let $I_0$ be a fixed finite set. Let $\fset_{I_0}^{\rm surj}$ be the category
of finite sets $I$ equipped with an \emph{arbitrary map} $\pi_0:I_0\to I$, where the morphisms
are maps under $I_0$ that are surjective. 

\medskip

As in Section  \ref{ss:defn of Ran}, for a separated scheme $X$ we consider the
functor
$$X^{(\fset_{I_0}^{\rm surj})^{\on{op}}}\longrightarrow \on{Sch}:(I,\pi_0:I_0\to I)\squig X^I,$$
and the corresponding functor
\begin{equation} \label{e:Ran diag I}
\fD^!(X^{(\fset_{I_0}^{\rm surj})^{\on{op}}}):\fset_{I_0}^{\rm surj}\longrightarrow  \ooCat^{\on{st}}_{\on{pres},\on{cont}}.
\end{equation}

\medskip

\begin{definition} \label{d:ran I}
The $\oo$-category $\fD(\ran_{I_0} X)$ is the limit of the functor in \eqref{e:Ran diag I} in $\ooCat^{\on{st}}_{\on{pres},\on{cont}}$.
\end{definition}

For $(I,\pi_0:I_0\to I)$ we let $(\Delta^I_{I_0})^!$ denote the tautological evaluation 
functor $$\fD(\ran_{I_0} X)\longrightarrow \fD(X^I).$$ For $I=I_0$ and $\pi_0=\id$, we will shall also use the notation
$(\Delta^{\on{main}}_{I_0})^!$. 

\medskip

As in Section  \ref{ss:defn of Ran}, we have a canonical equivalence
\begin{equation} \label{e:Ran as colimit I}
\fD(\ran_{I_0} X)\simeq \underset{(\fset_{I_0}^{\rm surj})^{\on{op}}}{\colim}\, \fD^*(X^{\fset_{I_0}^{\rm surj}}).
\end{equation}

For $(I,\pi_0:I_0\to I)$ we let $(\Delta^I_{I_0})_*$ denote the resulting
functor $\fD(\ran_{I_0} X)\to \fD(X^I)$, which is easily seen to be the left adjoint of 
$(\Delta^I_{I_0})^!$. 

\medskip

For $I=I_0$ and $\pi_0=\id$, we will shall also use the notation
$(\Delta^{\on{main}}_{I_0})_*$. It is easy to see that the adjunction map
$$\on{Id}\longrightarrow (\Delta^{\on{main}}_{I_0})^!\circ (\Delta^{\on{main}}_{I_0})_*$$
is a homotopy equivalence, i.e., the functor $(\Delta^{\on{main}}_{I_0})_*$ is
fully faithful.

\begin{definition}
We shall say that an object of $\fD(\ran_{I_0} X)$ is supported on $X^{I_0}$ if it lies
in the essential image of $(\Delta^{\on{main}}_{I_0})_*$.
\end{definition}

\sssec{}

We shall now introduce the actions of $\fD(\ran X)$ on $\fD(\ran_{I_0} X)$ in the $\star$ and chiral contexts.
We shall define the corresponding functors 
\begin{equation} \label{e:J fold tensor I}
\fD(\ran X)^{\otimes J}\otimes \fD(\ran_{I_0} X)\longrightarrow \fD(\ran_{I_0} X)
\end{equation}
in both contexts, in the style of Section  \ref{ss:defn of conv expl}. Upgrading them to 
the actual datum of action is done as in Section  \ref{ss:defn of conv}. 

\medskip

Using \eqref{e:Ran as colimit} and \eqref{e:Ran as colimit I}, to define a functor as in \eqref{e:J fold tensor I} it suffices
to define a functor
$$m_{J,I_0}:
\underset{J}{\underbrace{(\fset^{\rm surj})^{\on{op}}\times...\times  (\fset^{\rm surj})^{\on{op}}}}\times 
(\fset_{I_0}^{\rm surj})^{\on{op}} \longrightarrow (\fset_{I_0}^{\rm surj})^{\on{op}}$$
and a natural transformation between the resulting two functors
$$\underset{J}{\underbrace{(\fset^{\rm surj})^{\on{op}}\times...\times  (\fset^{\rm surj})^{\on{op}}}}\times
(\fset_{I_0}^{\rm surj})^{\on{op}}\rightrightarrows
\ooCat^{\on{st}}_{\on{pres}}:$$
\begin{equation} \label{e:nat transf for monoidal I}
\left((I_J,I_0\to I')\squig (\underset{j\in J}\otimes \fD(X^{I_j}))\otimes \fD(X^{I'})\right)\Rightarrow 
\left((I_J,I_0\to I')\squig \fD(X^{m_{J,I_0}(I_J,I_0\to I')})\right).
\end{equation}

Here $I_J$ has the same meaning as in Section  \ref{sss:defn of monoidal}, and $I_0\to I'$
is an object of $\fset_{I_0}$.

\sssec{}

In both cases we set
$$m_{J,I_0}(I_J,I_0\to I'):=I\sqcup I',$$
where $I:=\underset{j\in J}\sqcup\, I_j$, 
with the map
$$I_0\to I'\hookrightarrow I\sqcup I'.$$ 
We let $\pi$ denote the map $I\twoheadrightarrow J$ as in Section  \ref{ss:defn of conv}. Let
$\pi_{I_0}$ denote the map $I\sqcup I'\to J\sqcup \on{pt}$ that sends $I'$ to $\on{pt}$.

\medskip

For the $\star$ action, we define the functor of \eqref{e:nat transf for monoidal I} to be the external tensor
product 
$$\Bigl(M^{I_j}\in \fD(X^{I_j}),\,M^{I'}\in \fD(X^{I'})\Bigr)\squig \left((\underset{j}\boxtimes M^{I_j})\boxtimes M^{I'}
\in \fD(X^{I\sqcup I'})\right).$$

\medskip

For the chiral action we let the natural transformation \eqref{e:nat transf for monoidal I} to
be 
$$\left(M^{I_j}\in \fD(X^{I_j}),\,M^{I'}\in \fD(X^{I'})\right)\squig 
\left(\jmath(\pi_{I_0})_*\circ \jmath(\pi_{I_0})^*\bigl((\underset{j}\boxtimes M^{I_j})\boxtimes M^{I'}\bigr)\in \fD(X^{I\sqcup I'})\right).$$

\sssec{}

As in Lemma \ref{l:ch product on open}, for $M_j\in \fD(\ran X)$, $j\in J$ and $M'\in \fD(\ran_{I_0} X)$ we have
an explicit description of the object 
$$(\underset{j\in J}{\otimes^{\ch}}\, M_j)\otimes^{\ch} M'\in \fD(\ran_{I_0} X).$$ 
Namely, for $(I,\pi_0:I_0\to I)\in \fset_{I_0}$ we have a canonical
homotopy equivalence
\begin{equation} \label{e:ch product on open I}
(\Delta^{I}_{I_0})^!\left((\underset{j\in J}{\otimes^{\ch}}\, M_j)\otimes^{\ch} M'\right)\simeq
\underset{\pi_{\on{pt}}}\oplus\, \jmath(\pi_{\on{pt}})_*\circ 
\jmath(\pi_{\on{pt}})^*\left((\underset{j\in J}\boxtimes\, (\Delta^{I_j})^!(M_j))\boxtimes
(\Delta^{I'}_{I_0})^!(M')\right),
\end{equation}
where the direct sum is taken over the set of all surjections $\pi_{\on{pt}}:I\to J\sqcup \on{pt}$, such that 
$\pi_{\on{pt}}\circ \pi_0$ sends $I_0\to \on{pt}\in J\sqcup \on{pt}$, and where $I'\subset I$ equals
$(\pi_{\on{pt}})^{-1}(\on{pt})$. 

\medskip

As in Section  \ref{ss:proof of nilp}, the homotopy equivalence \eqref{e:ch product on open I} implies 
that the chiral action of $\fD(\ran X)$ on the module category $\fD(\ran_{I_0} X)$ is nilpotent.

\ssec{Chiral and factorization modules}

\sssec{}

For $A\in \Lie\alg^{\ch}(\ran X)$, we let $\m^{\ch}_A(\ran_{I_0} X)$ denote the resulting $\oo$-category of modules in 
$\fD(\ran_{I_0} X)$. We call its objects chiral $A$-modules on $\ran_{I_0} X$. 

\medskip

We shall denote by $\m^{\ch}_A(X^{I_0})\subset \m^{\ch}_A(\ran_{I_0} X)$ the full subcategory spanned
by objects supported on $X^{I_0}$. We shall call its objects chiral $A$-modules on $X^{I_0}$.

\medskip

Similarly, given $B\in \on{Com}\coalg(\ran _X)$, let $\Comod^{\ch}_B(\ran_{I_0} X)$ denote the $\oo$-category of
$B$-comodules in $\fD(\ran_{I_0} X)$. We shall call its objects chiral $B$-comodules on $\ran_{I_0} X$.

\medskip

From Theorem \ref{t:Koszul} and Corollary \ref{c:KD Lie for modules} we obtain:

\begin{cor} \label{c:chiral Koszul for modules} For any $A\in \Lie\alg^{\ch}(\ran X)$, the functor
$$\C_A^{\ch}:\m^{\ch}_A(\ran_{I_0} X)\longrightarrow \Comod^{\ch}_{\C^{\ch}(A)}(\ran_{I_0} X)$$
is an equivalence.
\end{cor}

\sssec{}

Let $B$ be an object of $\on{Com}\coalg^{\ch}(\ran X)$, and let $N$ be an object
of $\Comod^{\ch}_B(\ran_{I_0} X)$.  

\medskip

For an object $(I,\pi_0:I_0\to I)\in \fset_{I_0}$
and a map $\pi_{\on{pt}}:I\to J\sqcup \on{pt}$, from \eqref{e:ch product on open I} we obtain a map
$$(\Delta^I_{I_0})^!(N)\longrightarrow \jmath(\pi_{\on{pt}})_*\circ 
\jmath(\pi_{\on{pt}})^*\left((\underset{j\in J}\boxtimes\, (\Delta^{I_j})^!(B))\boxtimes
(\Delta^{I'}_{I_0})^!(N)\right),$$
and by adjunction a map
\begin{equation} \label{e:factorization map I}
\jmath(\pi_{\on{pt}})^*\left((\Delta^I_{I_0})^!(N)\right)\longrightarrow 
\jmath(\pi_{\on{pt}})^*\left((\underset{j\in J}\boxtimes\, (\Delta^{I_j})^!(B))\boxtimes
(\Delta^{I'}_{I_0})^!(N)\right).
\end{equation}

\medskip

Assume now that $B$ is a factorization coalgebra.

\begin{definition}
$N$ is a factorization $B$-comodule if, for every $(I,\pi_0:I_0\to I)$
and a map $\pi_{\on{pt}}:I\to J\sqcup \on{pt}$ as above, the map in \eqref{e:factorization map I} is a homotopy equivalence.
\end{definition}

We let $\Comod^{\on{Fact}}_B(\ran_{I_0} X)$ denote the full subcategory of $\Comod_B(\ran_{I_0} X)$ spanned
by factorization modules.

\medskip

As in Section  \ref{ss:factorization}, one shows:
\begin{cor}
Let $A$ be a chiral Lie algebra on $X$. The equivalence of Corollary \ref{c:chiral Koszul for modules} induces 
an equivalence between
the subcategory $\m_A(X^{I_0})$ of $\m_A(\ran_{I_0} X)$ spanned by modules supported on $X^{I_0}$,
and the subcategory of factorization $\C^{\ch}(A)$-comodules:
$$\m_A(X^{I_0})\simeq \Comod_{\C^{\ch}(A)}^{\on{Fact}}(\ran_{I_0} X).$$

\end{cor}


\begin{thebibliography}{99}

\bibitem[BD1]{bd} Beilinson, Alexander; Drinfeld, Vladimir. Chiral algebras. American Mathematical Society Colloquium Publications, 51. American Mathematical Society, Providence, RI, 2004. vi+375 pp.

\bibitem[BD2]{bd1} Beilinson, Alexander; Drinfeld, Vladimir. Quantization of Hitchin's integrable system and Hecke
eigensheaves. Available from http://math.uchicago.edu/$\sim$mitya/langlands.html$\sim$

\bibitem[BFN]{qcloops} Ben-Zvi, David; Francis, John; Nadler, David. Integral transforms and Drinfeld centers in derived algebraic geometry. J. Amer. Math. Soc. 23 (2010), no. 4, 909Ð966.

\bibitem[BF]{bzf} Frenkel, Edward; Ben-Zvi, David. Vertex algebras and algebraic curves. Second edition. Mathematical Surveys and Monographs, 88. American Mathematical Society, Providence, RI, 2004. xiv+400 pp.

\bibitem[BN]{BN} Ben-Zvi, David; Nadler, David. The character theory of a complex group. Preprint, 2009. arXiv:0904.1247


\bibitem[BFS]{BFS} Bezrukavnikov, Roman; Finkelberg, Michael; Schechtman, Vadim.
Factorizable sheaves and quantum groups. Lecture Notes in Mathematics 1691, Springer, 1998. 



\bibitem[BV]{bv} Boardman, J. Michael; Vogt, Rainer. Homotopy invariant algebraic structures on topological spaces. Lecture Notes in Mathematics, Vol. 347. Springer-Verlag, Berlin-New York, 1973. x+257 pp.

\bibitem[C1]{ching} Ching, Michael. Bar constructions for topological operads and the Goodwillie derivatives of the identity. Geom. Topol. 9 (2005), 833Ð933 (electronic). 

\bibitem[C2]{ching2} Ching, Michael. Bar-cobar duality for operads in stable homotopy theory. Preprint, 2010. Available from http://www.math.uga.edu/$\sim$mching/publications.html

\bibitem[CG]{kevinowen} Costello, Kevin; Gwilliam, Owen. Factorization algebras in perturbative quantum field theory. Preprint.

\bibitem[Dr]{Dr} Drinfeld, Vladimir. DG quotients of DG categories. J. Algebra 272 (2004), no. 2, 643Ð691.

\bibitem[F1]{thez} Francis, John. Derived algebraic geometry over $\cE_n$-rings. Thesis (PhD) Massachusetts Institute of Technology. 2008.

\bibitem[F2]{fact} Francis, John. Factorization homology of topological manifolds. In preparation.

\bibitem[Fr]{fresse} Fresse, Benoit. On the homotopy of simplicial algebras over an operad. Trans. Amer. Math. Soc. 352 (2000), no. 9, 4113Ð4141.

\bibitem[Ga1]{Ga} Gaitsgory, Dennis. Twisted Whittaker model and factorizable sheaves. Selecta Math., New ser. 13 (2007), 617--659. 

\bibitem[Ga2]{IndCoh} Gaitsgory, Dennis. Ind-coherent sheaves. Preprint, 2011. arXiv:1105.4857. 

\bibitem[GJ]{getzlerjones} Getzler, Ezra; Jones, Jon. Operads, homotopy algebra and iterated integrals for double loop spaces. Unpublished work, 1994.

\bibitem[GK]{koszul} Ginzburg, Victor; Kapranov, Mikhail. Koszul duality for operads. Duke Math. J. 76 (1994), no. 1, 203Ð272. 

\bibitem[GL:DG]{GL:DG} Notes on Geometric Langlands. DG Categories. \newline
Available from http://www.math.harvard.edu/$\sim$gaitsgde/GL/.

\bibitem[Go]{calc3} Goodwillie, Thomas. Calculus III: Taylor series. Geom. Topol. 7 (2003), 645Ð711.

\bibitem[Hi]{hinich} Hinich, Vladimir DG coalgebras as formal stacks. J. Pure Appl. Algebra 162 (2001), no. 2-3, 209Ð250.

\bibitem[Jo]{joyal} Joyal, Andr\'e. Quasi-categories and Kan complexes. Special volume celebrating the 70th birthday of Professor Max Kelly. J. Pure Appl. Algebra 175 (2002), no. 1-3, 207Ð222.

\bibitem[L1]{topos} Lurie, Jacob. Higher topos theory. Annals of Mathematics Studies, 170. Princeton University Press, Princeton, NJ, 2009. xviii+925 pp.

\bibitem[L2]{dag} Lurie, Jacob. Higher Algebra. Prepublication book draft. Available from http://www.math.harvard.edu/$\sim$lurie/

\bibitem[Mo]{moore} Moore, John. Differential homological algebra. Actes du Congr\`es International des Math\'ematiciens (Nice, 1970), Tome 1, Gauthier-Villars, Paris, 1971, pp. 335Ð339.

\bibitem[Q1]{quillen} Quillen, Daniel. Homotopical algebra. Lecture Notes in Mathematics, No. 43 Springer-Verlag, Berlin-New York 1967 iv+156 pp. 

\bibitem[Q2]{quillenrational} Quillen, Daniel. Rational homotopy theory. Ann. of Math. (2) 90 1969 205Ð295. 

\bibitem[Pr]{priddy} Priddy, Stewart. Koszul resolutions, Trans. Amer. Math. Soc. 152 (1970), 39Ð60.

\bibitem[Ro]{ro} Rozenblyum, Nick. Modules over a chiral algebra. Preprint, 2010. arXiv:1010.1998.

\end{thebibliography}
\end{document}